\documentclass[11pt]{article}

\usepackage{amsmath}
\usepackage{amsfonts}
\usepackage{amssymb}
\usepackage{amsthm}
\usepackage{graphicx}
\usepackage{enumerate}
\usepackage{xfrac}
\usepackage{tikz}
\newtheorem{thm}{Theorem}
\newtheorem{lem}{Lemma}
\newtheorem{prop}{Proposition}
\newtheorem{cor}{Corollary}
\newtheorem{conj}{Conjecture}

\newcommand{\F}{\mathbb F}
\numberwithin{equation}{section}

\begin{document}

\title{{\bf Legendre polynomials and complex multiplication, II: class numbers of quadratic fields and genus $2$ supersingular polynomials} }        
\author{Andrew Karl and Patrick Morton}        
\date{August 21, 2026}          
\maketitle

\begin{abstract}
The factorizations over $\mathbb{F}_p$ of two supersingular polynomials $h_p(x)$ and $g_p(x)$ for genus $2$ curves, discussed by Ibukiyama, Katsura and Oort in their 1986 paper, are investigated.  These polynomials are congruent modulo $p$ to the Jacobi polynomials $P_n^{(\alpha,0)}(1-2x)$, for $\alpha = \pm 1/4, \pm 1/6$,
respectively.  The number of their linear factors (mod $p$) is determined in terms of class numbers of the imaginary quadratic fields $\mathbb{Q}(\sqrt{-dp})$, 
 where $d \in \{1,2,3\}$.  The proofs use a quadratic transformation relating these polynomials to the Legendre polynomials $P_n(x)$; previous results on linear and binomial quadratic factors of $P_{(p-e)/4}(x)$ and $P_{(p-\bar e)/3}(x)$ proved by Brillhart and Morton; and properties of the irreducible quadratic factors of the class equations $H_{-3p}(X)$ and $H_{-12p}(X)$ (mod $p$). The quadratic transformation and linear factor results for $h_p(x) \equiv 
  P_n^{(\pm1/4,0)}(1-2x)$ were first discovered using artificial intelligence.
\end{abstract}

\section{Introduction.}

In this paper we will show that the results of the papers \cite{BM04}, \cite{Mort10} and \cite{Mor11} 
for Legendre polynomials can be applied to give similar results for two polynomials discussed by Ibukiyama, Katsura and Oort in \cite{IKO86}, 
which are closely related to Hasse-Witt invariants for several genus $2$ curves (see \cite{HW36}, \cite{Man61}).  The polynomials from \cite[p. 135]{IKO86} are
\begin{align}
\label{eqn:1.1} g_p(x) &= \sum_{j=0}^{[p/3]}{\binom{(p-1)/2}{[(p+1)/6]+j} \binom{(p-1)/2}{j}x^j},\\
\label{eqn:1.2} h_p(x) &= \sum_{j=0}^{[p/4]}{\binom{(p-1)/2}{[(p+1)/4]+j} \binom{(p-1)/2}{j}x^j}.
\end{align}
Here, $p \ge 5$ is a prime.  These polynomials are essentially Jacobi polynomials, and measure the supersingularity of the respective curves
\begin{align*}
C_\alpha: \ &y^2 = (x^3 - 1)(x^3 - \alpha), \ \alpha \neq 0, 1,\\
C_\beta: \ &y^2 = x(x^2-1)(x^2-\beta), \ \beta \neq 0,1,
\end{align*}
in characteristic $p$.  That is to say, $g_p(\alpha) = 0$ if and only if $C_\alpha$ is supersingular, and $h_p(\beta) = 0$ if and only if $C_\beta$ is supersingular,
 where a genus $2$ curve in characteristic $p$ is supersingular if and only if its Jacobian $J(C)$ has no points of order $p$. 
 The particular curves $C_\alpha, C_\beta$ are supersingular if and only if the Jacobian $A = J(C)$ is isomorphic to a product of two supersingular
  elliptic curves (\cite[pp. 133, 136]{IKO86}). \medskip
  
The connection between $h_p(x)$ and the Legendre polynomials considered in \cite{BM04} and \cite{Mort10} was initially discovered by a process involving
various AI platforms implemented by the first author Karl\footnote{A preliminary AI-generated document was produced by Karl, alternating work 
between (70 percent) OpenAI's GPT/Pro/Codex (5.2-5.6) and (30 percent) Anthropic's Claude Opus(4.6-4.8)/Fable(5) from February through July 2026.}, and then 
communicated to the second author.  Subsequently, Morton simplified, unified, and extended the proofs found by AI and added the connection between $g_p(x)$ 
and the Legendre polynomials discussed in \cite{Mor11}.  The results can be considered an interpretation for curves of genus $2$ of the main theorems 
of \cite{BM04}, \cite{Mort10} and \cite{Mor11} (with substantial help from \cite{Mor14}). \medskip
  
  The results discovered by Karl and AI for the polynomial $h_p(x)$ are summarized in the following theorem.  
  Let $\chi_p(a) = \left(\frac{a}{p}\right) \equiv a^{(p-1)/2}$ 
  be the quadratic character on $\mathbb{F}_p^\times$.
  \medskip
  
\begin{thm} \begin{enumerate}[(a)]
\item If $p \equiv 1$ (\textrm{mod} $4$), every root $\mu \in \mathbb{F}_p$ of $h_p(x)$ satisfies
$$\mu \neq 0,\ \ \left(\frac{\mu}{p}\right)_4 := \mu^{(p-1)/4} = -1.$$
Equivalently, every root of $h_p(x)$ in $\mathbb{F}_p$ is a square but not a fourth power.

\item If $p \equiv 1$ (\textrm{mod} $4$), then $\#\{\mu \in \F_p:h_p(\mu) = 0\} = \frac{h(-p)}{2}$,
where $h(-p)$ is the class number of the field $\mathbb{Q}(\sqrt{-p})$.

\item If $p \equiv 3$ (\textrm{mod} $4$), then
$$\#\{\beta \in \mathbb{F}_p^\times: h_p(\beta)=0, \ \chi_p(\beta)=-1\} = \frac{h(-2p)-2}{2},$$
where $h(-2p)$ is the class number of the quadratic field $\mathbb{Q}(\sqrt{-2p})$.

\item If $p \equiv 3$ \textrm{mod} $4$ and $p >3$, then
$$\#\{\beta \in \mathbb{F}_p^\times:h_p(\beta)=0, \chi_p(\beta)=+1\}
 =\begin{cases}
   3h(-p)-1,&p \equiv 3 \ \textrm{mod} \ 8,\\
   h(-p)-1,&p \equiv 7 \ \textrm{mod} \ 8.
 \end{cases}$$
 \end{enumerate}
 \label{thm:A}
\end{thm}

\begin{cor}
If $p \ge 5$ is a prime, the number of distinct linear factors of $h_p(x)$ (mod $p$) is equal to
\begin{equation*}
N_1(h,p) = \begin{cases} \frac{h(-p)}{2}, & p \equiv 1 \ \textrm{mod} \ 4,\\
3h(-p)+\frac{h(-2p)}{2}-2, & p \equiv 3 \ \textrm{mod} \ 8,\\
h(-p)+\frac{h(-2p)}{2}-2, & p \equiv 7 \ \textrm{mod} \ 8.
\end{cases}
\end{equation*}
\label{cor:A1}
\end{cor}

The result of this corollary is analogous to the results of \cite{BM04} (also see \cite{AT02}) for the Hasse invariants 
of the Legendre normal form, the Jacobi normal form, and the Deuring normal form of an elliptic curve.  It follows from  several explicit identities 
between $h_p(x)$ and the Legendre polynomial $P_{(p-e)/4}(x)$, using the counts of linear and binomial quadratic factors proved 
in \cite{BM04} and \cite{Mort10} for the Legendre polynomials $P_n(x)$ (mod $p$) and the related
polynomials
\begin{equation*}
W_n(x) = \sum_{k=0}^ n{\binom{n}{k}^2 x^k},
\end{equation*}
for $n = (p-1)/2$ and $n = (p-e)/4$, where $p = 4n+e, e \in \{1,3\}$. \medskip

In addition, we show that all of the roots of $h_p(x)$ (and of $g_p(x)$) in characteristic $p$ lie in the field $\mathbb{F}_{p^2}$ and are squares in that field, so that 
$h_p(x)$ factors as a product of linear and irreducible quadratic polynomials over $\mathbb{F}_p$.  Specific quadratics in this factorization 
can also be counted in terms of the same class numbers appearing in Corollary \ref{cor:A1}, as was discovered by the second author.

\begin{thm} The number $N_2(h,p)$ of irreducible quadratic factors of $h_p(x)$ over $\mathbb{F}_p$ of the form $x^2+ax+1$ is given by
\begin{equation*}
N_2(h,p) = 
\begin{cases} \frac{h(-2p)}{4}, & \ p \equiv 1 \ \textrm{mod} \ 8;\\
\frac{h(-2p)-2}{4}, & \ p \equiv 5  \ \textrm{mod} \ 8;\\
0, & \ p \equiv 3 \ \textrm{mod} \ 8;\\
\frac{h(-p)-1}{2}, & \ p \equiv 7 \ \textrm{mod} \ 8. \end{cases}
\end{equation*}
\label{thm:B}
\end{thm} 

These results for $h_p(x)$ also have application to the genus $3$ hyperelliptic curves $\mathcal{C}_\alpha: y^2 = (x^4-1)(x^4-\alpha)$ considered in \cite{Mta26}, since the polynomial $G_p(x)$ given there coincides with our $h_p(x)$.  Morita shows that $J(\mathcal{C}_\alpha)$ is a product of three supersingular elliptic curves in characteristic $p$ if and only if, in our notation, $W_{(p-1)/2}(\alpha) = h_p(\alpha) = 0$ in $\mathbb{F}_{p^2}$.    
\medskip

Theorem \ref{thm:A} motivated the second author to consider the polynomial $g_p(x)$ from \cite{IKO86}, and the count of its linear factors (mod $p$) 
yielded a similar connection to results of \cite{Mor11}.  In Section 4 we will prove the following.  We denote by $h(-3p)$ the class number of the field 
$\mathbb{Q}(\sqrt{-3p})$.

\begin{thm} \begin{enumerate}[(a)]

\item If $p \ge 7$ is a prime with $p \equiv 1$ (mod $3$), the number of distinct linear factors of $g_p(x)$ (mod $p$) is equal to
\begin{equation*}
N_1(g,p) = \begin{cases} \frac{h(-3p)}{2}, & p \equiv 1 \ \textrm{mod} \ 4,\\
0, & p \equiv 3 \ \textrm{mod} \ 4.
\end{cases}
\end{equation*}

\item If $p \ge 5$ is a prime with $p \equiv 2$ (mod $3$), the number of distinct linear factors of $g_p(x)$ (mod $p$) is equal to
\begin{equation*}
N_1(g,p) = \begin{cases} h(-p) + \frac{3h(-3p)}{2} - 2, & p \equiv 1 \ \textrm{mod} \ 8,\\
3h(-p) + \frac{h(-3p)}{2} - 2, & p \equiv 3 \ \textrm{mod} \ 8,\\
h(-p) + \frac{h(-3p)}{2} - 2, & p \equiv 5 \ \textrm{or} \ 7 \ \textrm{mod} \ 8.
\end{cases}
\end{equation*}
\end{enumerate}
\label{thm:C}
\end{thm}

There is also a result for $g_p(x)$ corresponding to Theorem \ref{thm:B}; see Theorem \ref{thm:15a} in Section 4.
Thus, for both polynomials $g_p(x)$ and $h_p(x)$, the counts of linear factors (and irreducible quadratic factors of a specific form)
involve a mixture of class numbers from different quadratic fields.  In particular, the formulas in the cases $p \equiv 3$ mod $4$ in 
Corollary \ref{cor:A1} are similar in shape to the last two formulas in part (b) of Theorem \ref{thm:C}, with the class number $h(-2p)$ 
being replaced by $h(-3p)$.  \medskip

The proof of Theorem \ref{thm:C} requires us to use, besides the counts of linear and 
binomial quadratic factors of the Legendre polynomial $P_{(p-e)/3}(x)$ from \cite{BM04} and \cite{Mor11}, 
facts about the class equations $H_{-3p}(X)$ and $H_{-12p}(X)$ from \cite{Mor14}.  
As in \cite[Thm. 1.2]{Mor14}, define
\begin{equation*}
K_{3p}(X) = \begin{cases}  H_{-3p}(X) H_{-12p}(X), \ & p \equiv 1 \ (\textrm{mod} \ 4),\\
H_{-12p}(X), \ & p \equiv 3 \ (\textrm{mod} \ 4). \end{cases}
\end{equation*}
Then for $p > 53$, $K_{3p}(X)$ satisfies the congruence (canonical factorization)
\begin{align}
\notag K_{3p}(X) &\equiv X^{2\delta_1}H_{-12}(X)^{2\delta_1} H_{-8}(X)^{4\delta_2} H_{-11}(X)^{4\delta_3} H_{-20}(X)^{4\delta_4} H_{-32}(X)^{4\delta_5}\\
\label{eqn:1.3}& \ \times H_{-35}(X)^{4\delta_6} \prod_{Q_3(r,s) \equiv 0(p)}{(X^2+rX+s)^2} \ \ (\textrm{mod} \ p);
\end{align}
here $H_{-d}(X)$ is the class equation for the discriminant $-d$ and the exponents $\delta_i \in \{0,1\}$ are determined by various Legendre symbols (see \eqref{eqn:4.8}, \eqref{eqn:4.9}, \eqref{eqn:4.12} and \eqref{eqn:4.13}).  Also, $Q_3(r,s) \in \mathbb{Z}[r,s]$ is a de-symmetrized form of the 
modular equation $\Phi_3(x,y)$, and $q(X) = X^2+rX+s$ runs over the irreducible quadratic factors (distinct from the $H_{-d}(X)$ in \eqref{eqn:1.3})
of the supersingular polynomial $ss_p(X)$ for which $Q_3(r,s) \equiv 0$ mod $p$.  When $p \equiv 1$ (mod $4$), separate congruences for $H_{-3p}(X)$ and $H_{-12p}(X)$ 
are also proved in \cite{Mor14} for $p >53$:
\begin{align}
\notag H_{-3p}(X) &\equiv H_{-12}(X)^{2\delta_1} H_{-8}(X)^{4\delta_2} H_{-20}(X)^{2\delta_4} H_{-32}(X)^{2\delta_5} \\
\label{eqn:1.4} & \ \ \times \prod_{i \in I}{(X^2+r_iX+s_i)^2} \ (\textrm{mod} \ p);\\
\notag H_{-12p}(X) &\equiv  X^{2\delta_1} H_{-11}(X)^{4\delta_3} H_{-20}(X)^{2\delta_4} H_{-32}(X)^{2\delta_5}H_{-35}(X)^{4\delta_6} \\
\label{eqn:1.5} & \ \ \times \prod_{j \in J}{(X^2+r_jX+s_j)^2} \ \ (\textrm{mod} \ p).
\end{align}
The indexing sets $I, J$ are disjoint, so that the products in these two congruences are relatively prime to each other and to the other factors
$H_{-d}(X)$.  The quadratics which appear, including $H_{-20}(X), H_{-32}(X), H_{-35}(X)$ (whenever their exponents $\delta_i = 1$), are always
irreducible (mod $p$).  Theorem \ref{thm:C} relies on the following remarkable fact, which we prove in Section 5.  Write $\bar f(X)$ for the polynomial $f(X)$
reduced modulo $p$.

\begin{thm}
\noindent Let $p > 53$ be a prime.  For each irreducible quadratic factor $q(X)$ of $K_{3p}(X)$ modulo $p$,
 there are $1$ or $2$ values of $a \in \mathbb{F}_p$, for which
\begin{equation*}
q(X) \equiv X^2+rX+s \equiv X^2+(a^3+126a^2+2944a)X+a(a-192)^3.
\end{equation*}
(a) If $p \equiv 3$ (mod $4$), $q(X) \mid \bar H_{-12p}(X)$ and each of these values satisfies
$$\left(\frac{a}{p}\right) = -1.$$
If $q(X) \notin \{\bar H_{-20}(X), \bar H_{-32}(X), \bar H_{-35}(X)\}$, 
$a$ is unique modulo $p$; otherwise there are two values for $a$. \medskip

 \noindent (b) If $p \equiv 1$ (mod $4$) and $q(X) \notin \{\bar H_{-20}(X), \bar H_{-32}(X), \bar H_{-35}(X)\}$, then the value of 
 $a$ defining $q(X)$ is unique modulo $p$, and
\begin{align*}
\left(\frac{a}{p}\right) &= -1, \ \textrm{if and only if} \ q(X) \mid H_{-3p}(X),\\
\left(\frac{a}{p}\right) &= +1, \ \textrm{if and only if} \ q(X) \mid H_{-12p}(X).
\end{align*}
If $q(X) \equiv H_{-20}(X)$ or $H_{-32}(X)$ (mod $p$), then there are two values of $a$, 
one which is a quadratic residue and one a nonresidue modulo $p$.  If $q(X) \equiv H_{-35}(X)$ (mod $p$),
then there are two values of a which define $q(X)$ and both are quadratic residues.
\label{thm:D}
\end{thm}
Thus, the quadratic character of $a$ mod $p$ can be used to determine which of the class equations $H_{-3p}(X)$ or $H_{-12p}(X)$ the 
polynomial $q(X)$ divides, when $p \equiv 1$ (mod $4$).  This criterion, which is easier to apply than the criterion given in \cite{Mor14},
follows from the latter (see \eqref{eqn:5.10}), as is shown in Section 5 of this paper.  Theorem \ref{thm:D} relies on the fact that the values
of $a$ are in 1--1 correspondence with irreducible factors $x^2+ax-27a$ of the polynomial $W_{(p-e)/3}(1-x/27)$ (mod $p$) ($e=1$ or $2$),
whose roots are $\alpha^3, \beta^3 = \alpha^{3p}$, where $\alpha \in \mathbb{F}_{p^2}$ and $(X,Y) = (\alpha,\alpha^p)$ is a solution of the Fermat equation
$$Fer_3: \ 27X^3 + 27Y^3 = X^3 Y^3.$$
Each such solution corresponds to an endomorphism $\mu$ of the supersingular elliptic curve
$$E_3(\alpha): \ Y^2+\alpha XY + Y = X^3$$
satisfying $\mu^2 = -3p$.  See \cite[\S 5]{Mor11}.  As part of the proof in Section 5, we show that the quadratic character of this element $\alpha \in \mathbb{F}_{p^2}^\times$ is
$$\psi_p(\alpha) := \alpha^{(p^2-1)/2} = \left(\frac{3}{p}\right) \ \textrm{or} \ -\left(\frac{3}{p}\right),$$
according as the endomorphism $\mu$ acts as a transposition or the identity permutation on the nontrivial points in $E_3(\alpha)[2]$.  See Theorems \ref{thm:16} and \ref{thm:17}.
\medskip

The plan of the paper is as follows.  Section 2 contains several results on the quadratic character of the roots of $W_{(p-1)/2}(x)$ and 
$W_{(p-3)/4}(x)$ modulo $p$ which are needed for the proof of Theorem 1.  Section 3 contains the development for the polynomial
$h_p(x)$ produced by the first author and AI, with modifications and additions by the second author, and includes the proofs
of Theorems \ref{thm:A} and \ref{thm:B}.  Section 4 contains the development for the polynomial $g_p(x)$ and the proof of Theorem \ref{thm:C},
assuming Theorem \ref{thm:D} (which we state as Conjecture \ref{conj:1} in Section 4.2, as it was discovered in the process of finding
a proof of Theorem \ref{thm:C}).  We also determine all primes $p$ for which $g_p(x)$ and $h_p(x)$ split into products of linear factors modulo $p$
(Corollary \ref{cor:7} and Proposition \ref{prop:10}).  In Section 5 we prove Theorem \ref{thm:D} (Conjecture \ref{conj:1}) by finding rational expressions 
for the points of order $2$ on the Deuring normal form.  \medskip

In Sections 2 and 3, we have explicitly marked with (AI) all the results that were found by Karl and AI, and with (M) the results and proofs contributed
 by the second author.  The remaining sections have been worked out by the second author without additional help from AI (except for a prime search
 for Proposition \ref{prop:10} and proofreading), although this development
 was certainly influenced and inspired by the results and connections discovered by AI and the first author.  
 All of the prose in this paper is written by the second author, except in some proofs (in Section 3) where we have followed the 
 AI-generated language more or less closely (sometimes modifying it for clarity, simplicity and elegance).  The notation we employ conforms 
 to the notation of \cite{BM04} and \cite{Mort10}.

\section{The roots of $W_{(p-1)/2}(x)$ and $W_{(p-3)/4}(x)$.}

Recall that $(-1)^{(p-1)/2}W_{(p-1)/2}(\lambda)$ is the Hasse invariant for the Legendre normal form

\begin{equation}
E_\lambda: \ y^2 = x(x-1)(x-\lambda).
\label{eqn:2.1}
\end{equation}

\begin{thm}[M, 2025] Let $p > 3$ be a prime number.
\begin{enumerate}[(a)]

\item If $p \equiv 3$ mod $8$, then none of the roots of
$$W_{(p-1)/2}(x) = \sum_{k=0}^{(p-1)/2}{\binom{(p-1)/2}{k}^2 x^k}$$
in $\mathbb{F}_p$ are squares in $\mathbb{F}_p$.

\item If $p \equiv 7$ mod $8$, then the number of roots of $W_{(p-1)/2}(x)$ in
$\mathbb{F}_p$ which are squares in $\mathbb{F}_p$ equals $2h(-p)$.
\end{enumerate}
\label{thm:5}
\end{thm}

The key to proving this theorem is that roots of $W_{(p-1)/2}(x)$ generally come in groups of $6$, which are orbits of the anharmonic group:
$$\lambda, \frac{1}{\lambda}, 1-\lambda, \frac{1}{1-\lambda}, \frac{\lambda-1}{\lambda}, \frac{\lambda}{\lambda - 1}.$$
The exception to this is when $j = 1728$ and $\lambda = 2, \frac{1}{2}, -1$.  In this case the orbit of $2$ has order $3$.  When $p \equiv 3$ mod $4$, the
$j$-invariant $j = 1728$ is always supersingular, so $\lambda = 2, \frac{1}{2}, -1$ will always be roots of $W_{(p-1)/2}(x)$ mod $p$.  For this orbit,
$2,  1/2$ are squares if $p \equiv 7$ (mod $8$) and non-squares if $p \equiv 3$ (mod $8$), by the supplement to quadratic reciprocity; while $-1$ is always a 
non-square when $p \equiv 3$ (mod $4$). \medskip

For the other orbits, we need the following.

\begin{lem}
Suppose $p \equiv 3 \pmod 4$. Either none of the elements
$$\left\{\lambda, \frac{1}{\lambda}, 1-\lambda, \frac{1}{1-\lambda}, \frac{\lambda-1}{\lambda}, \frac{\lambda}{\lambda - 1}\right\}$$
are squares in $\mathbb{F}_p$ or exactly $2/3$ of them are squares.
\label{lem:1}
\end{lem}

\begin{proof}
Suppose that $\lambda = u^2$ in $\mathbb{F}_p$.  If $1-\lambda$ is also a square, then we have the Legendre symbol
\begin{equation}
\left(\frac{(\lambda-1)/\lambda}{p}\right) = \left(\frac{-1}{p}\right) \left(\frac{1-\lambda}{p}\right) \left(\frac{\lambda}{p}\right) = -1.
\label{eqn:2.2}
\end{equation}
On the other hand, if $1-\lambda$ is not a square, then the same calculation shows that $(\lambda-1)/\lambda$ is a square.  The
remaining elements of the orbit are reciprocals of $\lambda, 1-\lambda, (\lambda-1)/\lambda$.  It follows that if one element of the orbit
is a square, then $4$ of them will be squares and $2$ will be non-squares. \medskip

On the other hand, there is no contradiction if all six of the elements in the orbit of $\lambda$ are non-squares, by \eqref{eqn:2.2}.
\end{proof}

\noindent {\it Proof of Theorem \ref{thm:5}.} \medskip

In order to prove Theorem \ref{thm:5}, we show that there is a square in every orbit of roots of $W_{(p-1)/2}(x)$ in $\mathbb{F}_p$, when
$p \equiv 7$ (mod $8$); and no element of any orbit is a square, if $p \equiv 3$ (mod $8$).  This will imply the theorem because of Lemma \ref{lem:1} and 
the fact that $W_{(p-1)/2}(x)$ has $3h(-p)$ linear factors (\cite{AT02} or \cite{BM04}), when $p \equiv 7$ (mod $8$), so 
$\frac{2}{3} \cdot 3h(-p) = 2h(-p)$ roots which are squares in $\mathbb{F}_p$.  \medskip

To prove this, note that the Legendre curve $E_\lambda$ in \eqref{eqn:2.1} has $p+1$ points in $\mathbb{F}_p$, 
when $\lambda$ is a root of $W_{(p-1)/2}(x)$, since $E_\lambda$ is supersingular in that case.  Thus,
$$|E_\lambda(\mathbb{F}_p)| = p+1$$
implies that
\begin{equation}
E_\lambda(\mathbb{F}_p)[4] \cong \mathbb{Z}_2 \oplus \mathbb{Z}_2, \ \ p \equiv 3 \ (\textrm{mod} \ 8);
\label{eqn:2.3}
\end{equation}
and
\begin{equation}
E_\lambda(\mathbb{F}_p)[4] \cong \mathbb{Z}_4 \oplus \mathbb{Z}_2 \ \textrm{or} \  \mathbb{Z}_4 \oplus \mathbb{Z}_4, \ \ p \equiv 7 \ (\textrm{mod} \ 8).
\label{eqn:2.4}
\end{equation}
(In fact, only the first possibility in \eqref{eqn:2.4} occurs; see below.)  Now use \cite[Thm 8.14, p. 211]{Wa08} 
(which holds over an arbitrary field of characteristic $p \neq 2$).  This says that on the curve
$$E: \ y^2 = (x-\alpha)(x-\beta)(x-\gamma),$$
the point $(\gamma,0)$ is the double of some point in $E(k)$ iff $\gamma-\alpha$ and $\gamma-\beta$ are squares in the field $k$.  Thus, $(0,0)$ is not the 
double of a point in $E_\lambda$ over $k = \mathbb{F}_p$, since $0-1 = -1$ is not a square in $\mathbb{F}_p$, in our case.  In case $p \equiv 3$ (mod $8$), 
\eqref{eqn:2.3} shows that none of the points of order $2$ on $E_\lambda$ is the double of a point, so
one of $\lambda, \lambda-1$ is not a square, and the 
same holds for $1, 1-\lambda$.  Hence, $1-\lambda$ is not a square, which implies $\lambda-1$ is a square, so that $\lambda$ is also not a square.  Then 
\eqref{eqn:2.2} implies that $(\lambda-1)/\lambda$ is also not a square.  This proves part (a) of Theorem \ref{thm:5}. \medskip

On the other hand, if $p \equiv 7$ (mod $8$), at least one of the points of order $2$ on $E_\lambda$, meaning either $(1,0)$ or $(\lambda,0)$ has to be the 
double of a point.  Hence, either $1, 1-\lambda$ are squares in $\mathbb{F}_p$ or $\lambda, \lambda-1$ are squares in $\mathbb{F}_p$.  Now Lemma \ref{lem:1}
shows that $2/3$ of the elements in every orbit must be squares.  This completes the proof.  $\square$

\begin{thm}[M] \begin{enumerate}[(a)]
\item If $p \equiv 3$ (mod $8$), all of the  $3h(-p)-1$ roots $\lambda \in \mathbb{F}_p$ of $W_{(p-3)/4}(x)$ are quadratic residues (mod $p$).
\smallskip

\item If $p \equiv 7$ (mod $8$), $h(-p)-1$ of the $2h(-p)-1$ roots $\lambda \in \mathbb{F}_p$ of $W_{(p-3)/4}(x)$ are quadratic residues 
(mod $p$), and $h(-p)$ of these roots are quadratic non-residues.
\end{enumerate}
\label{thm:6}
\end{thm}

\begin{proof} (M, AI)
(a) We use the congruence
\begin{equation}
W_{(p-1)/2}(x) \equiv (1-2x) W_{(p-3)/4}(4x(1-x)) \ \ (\textrm{mod} \ p)
\label{eqn:2.5}
\end{equation}
from \cite[Eq. (1.2)]{BM04}.  Let $\rho \in \mathbb{F}_p$ be a root of $W_{(p-3)/4}(x)$.  Then this congruence
implies that $\rho = 4\lambda (1-\lambda)$, for some root $\lambda \in \mathbb{F}_{p^2}$ of $W_{(p-1)/2}(x)$.  If $\lambda \in \mathbb{F}_p$,
 then Theorem \ref{thm:5}(a) implies that $\lambda$ and $1-\lambda$ are nonsquares in $\mathbb{F}_p$, so $\rho$ must be
  a square.  Suppose that $\lambda \in \mathbb{F}_{p^2}-\mathbb{F}_p$.  We use 
the fact from \cite[Thm. 5.4]{Mor06} that all the roots of $W_{(p-1)/2}(x)$ are $4$-th powers in $\mathbb{F}_{p^2}$ (first proved
by Landweber in 1988; see the references in \cite{Mor06}).  Then with $\lambda = \beta^4$, $\beta \in \mathbb{F}_{p^2}$, we have
$$\lambda^2-\lambda+\frac{\rho}{4} = 0  \ \Rightarrow \ \beta^{4p} = \lambda^p = 1- \lambda.$$
Hence $\rho = 4 \beta^4 \beta^{4p} = 4N_{\mathbb{F}_{p^2}/\mathbb{F}_p}(\beta)^4$ is certainly a 
square in $\mathbb{F}_p$.  This proves (a).  
(The argument for quadratic $\lambda \in \mathbb{F}_{p^2}$ comes from AI, which used the stronger result from \cite{AT02} that
 $\lambda = -\beta^8$, for $\beta \in \mathbb{F}_{p^2}$.) \medskip

(b) By the argument in the first proof in \cite[p. 105]{BM04} (for the case $p \equiv 7$ (mod $8$)), the roots $\rho$ of $W_{(p-3)/4}(x)$ 
come from $\frac{3h(-p)-1}{2}$ pairs $(\lambda, 1-\lambda)$ of roots in $\mathbb{F}_p$ of
$W_{(p-1)/2}(x)$, and $\frac{h(-p)-1}{2}$ pairs of conjugate quadratic roots of $W_{(p-1)/2}(x)$ satisfying $\lambda^2-\lambda + \frac{\rho}{4} = 0$.  
As in the proof of part (a), the roots $\rho$ coming from conjugate quadratic pairs are squares in $\mathbb{F}_p$.  
Suppose that $\lambda$ is a square in $\mathbb{F}_p$. 
Then two thirds of the elements in the orbit
$$\left\{\lambda, \frac{1}{\lambda}, 1-\lambda, \frac{1}{1-\lambda}, \frac{\lambda-1}{\lambda}, \frac{\lambda}{\lambda - 1}\right\}$$
are squares, and $\rho = 4\lambda(1-\lambda)$ is a square if and only if $1-\lambda$ is a square.  However, if $(\lambda,1-\lambda)$ is a pair of squares,
then one element in each of the pairs
$$(\frac{1}{\lambda}, \frac{\lambda-1}{\lambda}), \ \ (\frac{1}{1-\lambda}, \frac{\lambda}{\lambda-1})$$
is a nonsquare.  Subtracting off the pair $(-1,2)$, which yields the root $\rho = -8$ (a nonsquare in this case), gives that
$$\frac{1}{3} \cdot \frac{3h(-p)-3}{2} = \frac{h(-p)-1}{2}$$
of pairs of roots $(\lambda,1-\lambda)$ yield $\rho$'s which are squares.  Hence, there are a total of
$$\frac{h(-p)-1}{2} + \frac{h(-p)-1}{2} = h(-p)-1$$
roots of $W_{(p-3)/4}(x)$ in $\mathbb{F}_p$ which are squares in $\mathbb{F}_p$.  This proves (b).
\end{proof}

\noindent {\bf Remark.} Morton conjectured this result in 2023.

\section{The Ibukiyama-Katsura-Oort polynomials.}

We now introduce the following polynomials from \cite[p. 135]{IKO86}.  Define
\begin{align}
\label{eqn:3.1} g_p(x) &= \sum_{j=0}^{[p/3]}{\binom{(p-1)/2}{[(p+1)/6]+j} \binom{(p-1)/2}{j}x^j},\\
\label{eqn:3.2} h_p(x) &= \sum_{j=0}^{[p/4]}{\binom{(p-1)/2}{[(p+1)/4]+j} \binom{(p-1)/2}{j}x^j}.
\end{align}
In this section we discuss the polynomial \eqref{eqn:3.2}.  We discuss the polynomial in \eqref{eqn:3.1} in the
next section. \medskip

We write $p = 4n+e, \ e \in \{1,3\}$, for a prime $p > 3$ and we set $m = (p-1)/2 = 2n+s, s = \frac{e-1}{2} \in \{0,1\}$.  Then
\begin{equation}
h_p(x) =  \sum_{j=0}^{n}{\binom{m}{n+s+j} \binom{m}{j}x^j} = \sum_{j=0}^{n}{\binom{2n+s}{n-j} \binom{2n+s}{j}x^j}.
\label{eqn:3.3}
\end{equation}
We first write the polynomial $h_p(x)$ in terms of the Legendre polynomial.

\begin{prop}[AI] With notation as above, $p > 3$, and $n = \frac{p-e}{4}$,
\begin{equation}
h_p(u^2) \equiv (-1)^n u^n P_n\left(\frac{u+u^{-1}}{2}\right) = (-1)^n u^n P_n\left(\frac{u^2+1}{2u}\right)  \ (\textrm{mod} \ p).
\label{eqn:3.4}
\end{equation}
\label{prop:1}
\end{prop}

\begin{proof}
When $x=\frac{1}{2}(u+u^{-1})$, the generating function of the Legendre polynomials factors (see \cite[p. 115]{PS70}):
$$ \sum_{k \ge 0}P_k(x)z^k = (1-2xz+z^2)^{-1/2}
   =\left((1-uz)(1-u^{-1}z)\right)^{-1/2}.$$
Expanding each factor with $(1-w)^{-1/2} = \sum_r{(-1)^r{\binom{-1/2}{r}}w^r} =  \sum_r{4^{-r}\binom{2r}{r}w^r}$ and reading
off the coefficient of $z^n$ gives the identity
$$u^{n}P_n\left(\frac{u^2+1}{2u}\right) = 4^{-n} \sum_{k=0}^{n}{\binom{2k}{k} \binom{2n-2k}{n-k}u^{2k}} \ \  \textrm{in} \ \mathbb{Q}[u].$$
Modulo $p=4n+e$ one has $2n + s \equiv \frac{p-1}{2} \equiv -\frac{1}{2}$ (mod $p$), hence
\begin{align*}
\binom{2n + s}{k} &\equiv (-1)^k \frac{(\frac{1}{2})_k}{k!} \ (\textrm{mod} \ p) \ \ \textrm{and}\\
\binom{2n + s}{n-k} &\equiv (-1)^{n-k} \frac{(\frac{1}{2})_{n-k}}{(n-k)!} \ (\textrm{mod} \ p);
\end{align*}
where $(a)_k = a(a+1) \cdots (a+k-1)$ is the rising factorial.
Combining these with
$\binom{2r}{r}=4^{r}\frac{(\frac{1}{2})_r}{r!}$ from \cite[Eq. (5.37), p. 186]{GKP94} gives
$$4^{-n}\binom{2k}{k}\binom{2n-2k}{n-k} = \frac{(\frac{1}{2})_k(\frac{1}{2})_{n-k}}{k! (n-k)!}
   \equiv (-1)^{n}\binom{2n+s}{k} \binom{2n+s}{n-k} \ \textrm{mod} \ p.$$
Now \eqref{eqn:3.3} implies that
$$u^{n}P_n\left(\frac{u^2+1}{2u}\right)  \equiv (-1)^n h_p(u^2) \ (\textrm{mod} \ p).$$
\end{proof}

\noindent {\bf Remark.} It is easy to see using \eqref{eqn:3.4} that the polynomial $h_p(x)$ has distinct roots modulo $p$. Namely, if $h_p(x)$ has a multiple
root, then so does $h_p(u^2)$.  Hence there is a value $u$ in some extension of $\mathbb{F}_p$, for which $0 = h_p(u^2) = 2uh_p'(u^2)$.  
Then modulo $p$, we have
\begin{align*}
0 &\equiv u^n P_n\left(\frac{u+u^{-1}}{2}\right) = \frac{d}{dx} \big[x^nP_n\left(\frac{x+x^{-1}}{2}\right)\big]|_{x=u}\\
&= nu^{n-1} P_n\left(\frac{u+u^{-1}}{2}\right) + \frac{u^n}{2}  \left(1-\frac{1}{u^2}\right) P_n'\left(\frac{u+u^{-1}}{2}\right)\\
& \equiv \frac{(-1)^n n}{u} h_p(u^2) + \frac{u^n}{2}  \left(1-\frac{1}{u^2}\right) P_n'\left(\frac{u+u^{-1}}{2}\right)\\
&\equiv \frac{u^n}{2}  \left(1-\frac{1}{u^2}\right) P_n'\left(\frac{u+u^{-1}}{2}\right).
\end{align*}
Since $u \notin \{0, 1, -1\}$ (see \eqref{eqn:3.5} below), this would imply that $P_n(x)$ and $P_n'(x)$ have a root in common, which is false, by \cite[Eq. (A.3), p. 109]{BM04}.

\begin{cor}[M]
The polynomial $h_p(u^2)$ has the representation
\begin{equation}
h_p(u^2) \equiv (-4)^{-n} (u+1)^{2n} W_{(p-e)/4}\left(\left(\frac{u-1}{u+1}\right)^2\right) \ (\textrm{mod} \ p).
\label{eqn:3.5}
\end{equation}
\label{cor:2}
\end{cor}

\begin{proof}
We use the fact that
\begin{equation}
W_n(z) = (1-z)^n P_n\left(\frac{1+z}{1-z}\right),
\label{eqn:3.6}
\end{equation}
(see \cite[p. 80]{BM04} or \cite[VI, Problem 85.]{PS70}), from which we obtain that
$$(z+1)^nW_n\left(\frac{z-1}{z+1}\right) = 2^n P_n(z).$$
Then, still with $n = (p-e)/4$ and $z = \frac{u^2+1}{2u}$, we have $\frac{z-1}{z+1} = \left(\frac{u-1}{u+1}\right)^2$, and
\begin{align*}
(-1)^n u^n P_n\left(\frac{u^2+1}{2u}\right) &= (-2)^{-n} u^n \left(\frac{u^2+1}{2u}+1\right)^n W_n\left(\frac{z-1}{z+1}\right)\\
&= (-4)^{-n} (u+1)^{2n} W_n\left(\left(\frac{u-1}{u+1}\right)^2\right) \ (\textrm{mod} \ p).
\end{align*}
\end{proof}

\begin{cor}[M]
For any prime $p >3$, the roots of $h_p(x)$ lie in $\mathbb{F}_{p^2}$, so $h_p(x)$ factors (mod $p$) as a product of 
linear and quadratic factors.  Moreover, the roots of $h_p(x)$ are squares in $\mathbb{F}_{p^2}^\times$.
\label{cor:3}
\end{cor}

\begin{proof}
Corollary \ref{cor:2} implies that any root $\mu$ of $h_p(x)$ satisfies $\mu = u^2$, where $u \neq 0, \pm 1$ and
$v = \left(\frac{u-1}{u+1}\right)^2$ is a root of $W_{(p-e)/4}(x)$.  Now by \cite[Eq. (1.2)]{BM04}, $v = 4\lambda(1-\lambda)$, 
for some root $\lambda \neq \frac{1}{2}$ of $W_{(p-1)/2}(x)$.  By \cite[Thm. 5.4]{Mor06} the roots $\lambda$ of $W_{(p-1)/2}(x)$ 
are fourth powers in $\mathbb{F}_{p^2}$.   Since $1-\lambda$ is also a root of $W_{(p-1)/2}(x)$,  
$v$ is a square in $\mathbb{F}_{p^2}$.  Thus, $\frac{u-1}{u+1} \in \mathbb{F}_{p^2}$, giving that $u, \mu \in \mathbb{F}_{p^2}$.
\end{proof}

\subsection{Primes $p \equiv 1$ mod $4$.}

\begin{prop}[AI]
If $p \equiv 1$ (mod $4$) and $2n = \frac{p-1}{2}$,
\begin{equation}
(1+t)^{4n}W_{2n}\left(-\left(\frac{1-t}{1+t}\right)^2\right) \equiv (-1)^n h_p(t^4) \ (\textrm{mod} \ p).
\label{eqn:3.7}
\end{equation}
\label{prop:2}
\end{prop}

\begin{proof} (M) The original AI proof of this proposition used a hypergeometric representation of the polynomial $h_p(x)$.  We can 
do this more simply by noting that the roots of $W_{(p-1)/2}(x)$ are invariant under the substitutions of the anharmonic group; see Lemma \ref{lem:1}.
Put $p = 4n+1, u = t^2, r = \frac{1-t}{1+t}$.  Then applying the map $\lambda \rightarrow \frac{1}{1-\lambda}$, with $\lambda = -r^2$, yields 
(since both sides have leading coefficient $1$)
\begin{align*}
(1+t)^{4n}W_{2n}(-r^2) &\equiv (1+t)^{4n} (1+r^2)^{2n} W_{2n}\left(\frac{1}{1+r^2}\right)\\
&= 4^n (t^2+1)^{2n} W_{2n}\left(\frac{1}{1+r^2}\right) \ (\textrm{mod} \ p).
\end{align*}
Now we use the congruence
\begin{equation}
W_{(p-1)/2}(x) \equiv W_{(p-1)/4}(4x(1-x)) \ \ (\textrm{mod} \ p), \ p \equiv 1 \ (\textrm{mod} \ 4)
\label{eqn:3.8}
\end{equation}
from \cite[p. 81]{BM04} to write this as
\begin{align*}
(1+t)^{4n}W_{2n}(-r^2) &\equiv 4^n(t^2+1)^{2n} W_n\left(\frac{4}{1+r^2}(1-\frac{1}{1+r^2})\right)\\
&= 4^n(t^2+1)^{2n} W_n\left(\left(\frac{t^2-1}{t^2+1}\right)^2\right)\\
& = 4^n(u+1)^{2n} W_n\left(\left(\frac{u-1}{u+1}\right)^2\right)\\
&\equiv (-1)^n 4^{2n} h_p(u^2) \equiv (-1)^n h_p(t^4) \ (\textrm{mod} \ p),
\end{align*}
by \eqref{eqn:3.5}.  For this we have used that $4^{2n} \equiv 2^{4n} \equiv 2^{p-1} \equiv 1$ (mod $p$).
\end{proof}

\medskip

\begin{thm}[AI]
If $p \equiv 1$ (mod $4$), every root $\mu \in \mathbb{F}_p$ of $h_p(x)$ satisfies $\mu \neq 0,1$ and
$\mu^{(p-1)/4}=-1$.  Equivalently, each root in the prime field is a nonzero square that is not a fourth power.
\label{thm:7}
\end{thm}

\begin{proof}
It is easy to see that neither $0$ nor $1$ is a root.  For example, \eqref{eqn:3.7} gives that $(-1)^n h_p(1) = 2^{4n} W_{2n}(0) \equiv 1$ (mod $p$).
Furthermore, it is clear that $h_p(0) = \binom{2n}{n} \not \equiv 0$ (mod $p$).
\medskip

Now fix a root $\mu \in \mathbb{F}_p - \{0,1\}$.  Choose $\eta \in \overline{\mathbb{F}}_p$ with
$\eta^4=\mu$; then $\eta \neq0, \pm1, \pm i$.  Put
$$r=\frac{1-\eta}{1+\eta}, \ \ v = -r^2.$$
It is clear that $v \notin \{0,1\}$: $v = 0$ forces $\eta=1$; while $v=1$ forces $r^2=-1$, so $r=\pm i$ and
$\eta = \frac{1-r}{1+r} = \pm i$, both of which are excluded.  Because $h_p(\eta^4)=0$ and
$(1+\eta)^{4n}\neq0$, Proposition \ref{prop:2} evaluated at $t=\eta$ gives
$W_{2n}(v)=0$; thus, $E_v$ is a supersingular Legendre curve.  By \cite[Prop. 2.2]{AT02} or \cite[Prop. 1]{BM04}, $v \in \mathbb{F}_{p^2}$. \medskip

\emph{$\mu$ is a square.}  Suppose, instead, that $\mu$ is a nonsquare.  Then
$\zeta=\eta^{p-1}$ satisfies $\zeta^2=\mu^{(p-1)/2}=-1$, so $\zeta=\pm i\in\F_p$
and
$$\eta^{p^2}=\zeta^{p+1}\eta=\zeta^2\eta=-\eta,$$whence $r^{p^2}=\tfrac{1+\eta}{1-\eta}=r^{-1}$ and
$v^{p^2}=v^{-1}$.  But $v \in \F_{p^2}$ forces $v^{p^2}=v$, so $v^2=1$.  If $v=-1$, then $r^2=1$;
$r=1$ gives $\eta=0$, while $r=-1$ is impossible, since it would force $1=-1$.  If $v=1$, then
$r^2=-1$, so $r=\pm i \in \F_p$ and hence $\eta=\frac{1-r}{1+r} \in \F_p$,
contradicting $\eta^{p^2}=-\eta$.  Thus $\mu$ must be a square in $\mathbb{F}_p$. \medskip

\emph{$\mu$ is not a fourth power.}  Suppose instead $\mu\in(\F_p^\times)^4$.
Then $\eta$ may be taken in $\F_p$, so $v \in \F_p$ and $E_v/\F_p$ is
a supersingular Legendre curve over the prime field.  But then $\#E_v(\F_p)=p+1\equiv 2$ mod $4$.  This contradicts the full
rational $2$-torsion of a Legendre curve, which forces $4 \mid \#E_v(\F_p)$.  (See also \cite[Cor. 3]{BM04}.) \medskip

Thus, every root $\mu$ is a square but not a fourth power, i.e., $\mu^{(p-1)/4} = -1$.
\end{proof}

\begin{lem}[AI]
Every root $\rho \in \mathbb{F}_p$ of $W_{(p-1)/4}(x)$ is a nonzero square in $\mathbb{F}_p$.
\label{lem:2}
\end{lem}

\begin{proof}
This follows by the argument used in the proof of the second half of Theorem \ref{thm:6}(a).  This is because the roots of $W_{(p-1)/4}(x)$ 
are related to the roots $\lambda$ of $W_{(p-1)/2}(x)$ by $\rho = 4\lambda (1-\lambda)$, by virtue of \eqref{eqn:3.8},
 and none of the roots $\lambda$ lies in $\mathbb{F}_p$, by \cite[Thm. 1(a)]{BM04}.
\end{proof}

\begin{thm}[AI]
For every prime $p \equiv 1 \ \textrm{mod} \ 4$,
$$N_1(h,p) = \#\{\mu \in \F_p: h_p(\mu) = 0\}=\frac{h(-p)}{2},$$
where $h(-p)$ is the class number of $\mathbb{Q}(\sqrt{-p})$.
\label{thm:8}
\end{thm}

\begin{proof} (AI, M)
By Theorem \ref{thm:7} every root of $h_p(x)$ in $\mathbb{F}_p$ is a nonzero square in $\mathbb{F}_p$ different from $0$ or $1$.  
By Corollary \ref{cor:2}, for every $u \in \mathbb{F}_p -\{0,1,-1\}$, we have the equivalence
\begin{equation}
h_p(u^2) \equiv 0 \ \iff \ (u+1)^{2n} W_n\left(\left(\frac{u-1}{u+1}\right)^2\right) \equiv 0 \ (\textrm{mod} \ p).
\label{eqn:3.9}
\end{equation}
By Lemma \ref{lem:2}, every root $\rho \in \mathbb{F}_p^\times$ of $W_n(x)$ has the 
form $\rho = \left(\frac{a-1}{a+1}\right)^2$ with $a \in \mathbb{F}_p$, since $\rho \neq 1$.  Hence, \eqref{eqn:3.9} sets up a 1-1 correspondence
between the roots $\beta^2$ of $h_p(x)$ (which are never $0$ or $1$) and the roots $\rho = \left(\frac{a-1}{a+1}\right)^2$ of $W_n(x) = 0$.  
 We can see this as follows.  Considering the first congruence of \eqref{eqn:3.9}, a factor $x-\beta^2$ of $h_p(x)$, with $\beta \in \mathbb{F}_p$ 
 and $\beta \neq \pm 1$, yields two factors $(u-\beta)(u+\beta)$ of $h_p(u^2)$ on setting $x = u^2$.  Thus, there is a 1-2 map of factors of $h_p(x)$ 
 whose roots are squares to linear factors of $h_p(u^2)$. Conversely, every linear factor of $h_p(u^2)$ over $\mathbb{F}_p$ arises from 
 a linear factor of $h_p(x)$. \medskip
 
 Note that $h_p(x)$ is clearly palindromic by \eqref{eqn:3.3}, so that $\beta, \beta^{-1}$ are both roots of $h_p(x^2) = 0$, if one of them is.  Thus, the roots 
 of the second congruence of \eqref{eqn:3.9} come in reciprocal pairs.  Furthermore, a linear factor $(x-\rho)$ of $W_n(x)$ with 
 $\rho =\left(\frac{a-1}{a+1}\right)^2  \in (\mathbb{F}_p^\times)^2$ also yields two linear factors of $(u+1)^{2n} W_n\left(\left(\frac{u-1}{u+1}\right)^2\right)$, since
\begin{align*}
 (u+1)^2\bigg[\left(\frac{u-1}{u+1}\right)^2 -  \left(\frac{a-1}{a+1}\right)^2\bigg] &= \frac{4}{(a+1)^2}(au - 1)(u-a)\\
 &= \frac{4a}{(a+1)^2}\left(u-\frac{1}{a}\right)(u-a).
 \end{align*}
 Thus there is a 1-2 map of linear factors of $W_n(x)$ whose roots are squares to linear factors of the right side of \eqref{eqn:3.5}.
It follows that the number of roots of $h_p(x)$ in $\mathbb{F}_p$ (which are squares) equals the number of roots of $W_n(x)$ in $\mathbb{F}_p$
 (which are squares), and this number equals $\frac{1}{2}h(-p)$, by \cite[Thm. 1(c)]{BM04}.
\end{proof}

\begin{thm}[M]
If $p \equiv 1$ (mod $4$), the number $N_2(h, p)$ of irreducible quadratic factors of $h_p(x)$ over $\mathbb{F}_p$ of the form $x^2+ax+1$ is
\begin{equation}
N_2(h,p) = \begin{cases} \frac{h(-2p)}{4}, & \textrm{if} \ p \equiv 1 \ (\textrm{mod} \ 8),\\
\frac{h(-2p)-2}{4}, & \textrm{if} \ p \equiv 5  \ (\textrm{mod} \ 8);
\end{cases}
\label{eqn:3.10}
\end{equation}
where $h(-2p)$ is the class number of $\mathbb{Q}(\sqrt{-2p})$.
\label{thm:9}
\end{thm}

\begin{proof}
We will show that irreducible factors of $h_p(x)$ of the form $x^2+ax+1$ are in 1--1 correspondence with the irreducible factors of the form 
$x^2+a$, i.e. binomial quadratic factors, of $P_n(x)$ (mod $p$).  For each of the latter factors, substituting $x=\frac{u^2+1}{2u}$ gives that
\begin{equation}
4u^2(x^2+a) = u^4 + (4a+2)u^2 + 1,
  \label{eqn:3.11}
\end{equation}
which is a factor of $h_p(u^2)$, by  Proposition \ref{prop:1}.  Hence, there is a factor $x^2+(4a+2)x+1$ of $h_p(x)$ for each
 binomial quadratic factor $x^2+a$ of $P_n(x)$.  We claim that $x^2+(4a+2)x + 1$ is irreducible over $\mathbb{F}_p$.  Its discriminant 
 is $d = 16a(a+1)$.  Following an argument of AI, we show that $a+1$ is a square in $\mathbb{F}_p^\times$.
 For this we use the formula from \cite[Eq. (4.5)]{Mort10}, which says that for every binomial quadratic factor $x^2 +a$ of $P_n(x)$, the following
  factorization holds:
  \begin{equation}
  z^4 -\frac{4}{1+a}z^2+\frac{4}{1+a} = (z^2 + \rho z + \sigma)(z^2 - \rho z + \sigma) \ \textrm{over} \ \mathbb{F}_p.
\label{eqn:3.12}
  \end{equation}
 By this formula, $\frac{4}{1+a} = \sigma^2$ is a square in $\mathbb{F}_p^\times$.  Hence, $a+1$ is a square.  Now $-a$ is a nonsquare, since
 $x^2+a$ is irreducible; and since $p \equiv 1$ (mod $4$), $a$ is also a nonsquare.  It follows that  $d = 16a(a+1)$ is a nonsquare.  This shows that every
  irreducible factor $x^2+a$ of $P_n(x)$ gives rise to an irreducible factor $x^2+(4a+2)x+1$ of $h_p(x)$.  \medskip
  
 Equation \eqref{eqn:3.11} shows that every factor $x^2+(4a+2)x+1$ arises from some factor $x^2+a$ in this way, but we must check that the 
 corresponding $x^2+a$ is irreducible.  This holds because linear factors of $P_n(x)$ over $\mathbb{F}_p$ map by means of 
 \eqref{eqn:3.4} to linear factors of $h_p(x)$, for the following reason.  By Lemma \ref{lem:2}, every root of $W_n(x)$ in 
 $\mathbb{F}_p$ is a square in $\mathbb{F}_p$.  Hence, the displayed formula following \eqref{eqn:3.6} implies that every root $z \in \mathbb{F}_p$ 
 of $P_n(x)$ satisfies 
 $$\frac{z-1}{z+1} = b^2, \ b \in \mathbb{F}_p \ \Rightarrow \ z = \frac{1+b^2}{1-b^2}.$$
 But then
 $$2u \left(\frac{u^2+1}{2u} - z\right) = \left(u + \frac{b - 1}{b+1}\right) \left(u + \frac{b + 1}{b-1}\right).$$
 Since $P_n(-x) = (-1)^nP_n(x)$, the roots of $P_n(x)$ are invariant under $z \rightarrow -z$.  Hence the factor $(x+z)$ yields
 $$2u \left(\frac{u^2+1}{2u} + z\right) = \left(u - \frac{b - 1}{b+1}\right) \left(u - \frac{b + 1}{b-1}\right).$$
 Thus, when $z \neq 0$, the product $x^2-z^2$ of linear factors yields the factors
 $$u^2 -\frac{(b-1)^2}{(b+1)^2} \ \ \textrm{and} \ \ u^2 - \frac{(b+1)^2}{(b-1)^2}$$
 of $h_p(u^2)$, and therefore the linear factors
 $$x - \frac{(b-1)^2}{(b+1)^2} \ \ \textrm{and} \ \ x- \frac{(b+1)^2}{(b-1)^2}$$
 of $h_p(x)$.  The root $z=0$ yields the factor $x+1$ of $h_p(x)$.  
 This shows that linear factors of $P_n(x)$ only correspond to linear factors of $h_p(x)$ by means of \eqref{eqn:3.4}. \medskip
  
Therefore, if $x^2+(4a+2)x+1$ is an irreducible factor of $h_p(x)$, the factor $x^2+a$ of $P_n(x)$ in \eqref{eqn:3.11} must also be
irreducible.  Since the correspondence $x^2+a \longleftrightarrow x^2+(4a+2)x + 1$ is obviously 1--1, $N_2(h,p)$ is equal to the 
number of binomial quadratic factors of $P_n(x)$ (mod $p$), which is \eqref{eqn:3.10}, by \cite[Thm. 1.1]{Mort10}.
\end{proof}

\noindent {\bf Examples.}  For the primes $p = 29, 89$ we have
\begin{align*}
h_{29}(x) &\equiv 10(x + 1)(x + 16)(x + 20)(x^2 + x + 7)(x^2 + 25x + 25) \ (\textrm{mod} \ 29)\\
h_{89}(x) &\equiv 10(x + 17)(x + 21)(x + 71)(x + 79)(x + 80)(x + 84)\\
& \ \ \times (x^2 + 48x + 1)(x^2 + 76x + 1) (x^2 + 8x + 78) (x^2 + 26x + 87)\\
& \ \ \times (x^2 + 64x + 8)(x^2 + 76x + 44)(x^2 + 78x + 11)\\
& \ \ \times (x^2 + 88x + 81) \ (\textrm{mod} \ 89).
\end{align*}
We know $h(-29) = 6$ and $h(-58) = 2$, so the factorization of $h_{29}(x)$ agrees with the results of Theorems \ref{thm:8} 
and \ref{thm:9}.  Also, $h(-89) = 12$ and $h(-178) = 8$ so the same holds for $h_{89}(x)$. \medskip

We note the following representation of $h_p(x)$ as a Jacobi polynomial.

\begin{prop}[AI, M]
If $p \equiv e$ (mod $4$) and $n = \frac{p-e}{4}$, then
$$h_p(x) \equiv \begin{cases} P_n^{(-1/4,0)}(1-2x) \ (\textrm{mod} \ p), & p \equiv 1 \ (\textrm{mod} \ 4);\\		
P_n^{(1/4,0)}(1-2x) \ (\textrm{mod} \ p), & p \equiv 3 \ (\textrm{mod} \ 4).\end{cases}$$
\label{prop:3}
\end{prop}

\begin{proof}
First assume $p \equiv 1$ (mod $4$).  The Jacobi polynomial $u = P_n^{(\alpha,\beta)}(x)$ can be defined by
\begin{equation*}
P_n^{(\alpha,\beta)}(x) = \frac{(\alpha+1)_n}{n!} F_{2,1}(-n, 1+\alpha + \beta +n, \alpha+1;\frac{1}{2}(1-x)),
\end{equation*}
where $F_{2,1}$ is the hypergeometric function, so that
$$P_n^{(-1/4,0)}(1-2x) \equiv \frac{(3/4)_n}{n!} F_{2,1}(1/4, 1/2, 3/4; x) \ (\textrm{mod} \ p).$$
According to \cite[p. 139]{IKO86}, we also have
$$h_p(x) \equiv  c F_{2,1}(1/4, 1/2, 3/4; x) \ (\textrm{mod} \ p),  \ \ c = \binom{2n}{n}.$$
Hence, $P_n^{(-1/4,0)}(1-2x)$ and $h_p(x)$ are congruent, up to multiplying by a constant.
The leading coefficients of $h_p(x)$ and $P_n^{(-1/4,0)}(1-2x)$ are
\begin{align*}
\textrm{l.c.} \ h_p(x) &= \binom{2n}{n}, \ n = \frac{p-1}{4},\\
\textrm{l.c.} \ P_n^{(-1/4,0)}(1-2x) &= (-2)^n 2^{-n} \sum_{k=0}^n{\binom{n-1/4}{k} \binom{n}{n-k}}\\
 &= (-1)^n \binom{2n-1/4}{n} \equiv (-1)^n\binom{-3/4}{n} \ (\textrm{mod} \ p).
 \end{align*} 
We just need to show that $\binom{2n}{n} \equiv (-1)^n\binom{-3/4}{n} \ (\textrm{mod} \ p)$, for $ p = 4n+1$.
Since $2n \equiv -\frac{1}{2}$ (mod $p$), this is equivalent to 
$$\left(\frac{1}{2}\right)_n \equiv (-1)^n \left(\frac{3}{4}\right)_n  \ (\textrm{mod} \ p)$$
or
$$2^{-n} (1\cdot 3 \cdot 5 \cdots (2n-1)) \equiv (-1)^n 4^{-n} (3 \cdot 7 \cdot 11 \cdots (4n-1)) \ (\textrm{mod} \ p);$$
this is further equivalent to
$$2 \cdot 6 \cdot 10 \cdots (4n-2) \equiv (p-3)(p-7)(p-11) \cdots (p-4n+1) \ (\textrm{mod} \ p),$$
which is clear.  A similar argument proves the congruence when $p \equiv 3$ (mod $4$), using that 
\begin{align*}
h_p(x) &\equiv cF_{2,1}(3/4,1/2, 5/4; x),\ \ c = \binom{2n+1}{n},\\
P_n^{(1/4,0)}(1-2x) &\equiv \frac{(5/4)_n}{n!} F_{2,1}(3/4,1/2, 5/4; x),\\
\left(\frac{1}{2}\right)_n &\equiv (-1)^n \left(\frac{5}{4}\right)_n \ (\textrm{mod} \ p), \ p = 4n+3.
\end{align*}
\end{proof}
Thus, the counts in Theorems \ref{thm:8} and \ref{thm:9} carry over to the factors of the Jacobi polynomials $P_{\frac{p-1}{4}}^{(-1/4,0)}(x)$
modulo $p$, where the factors $x^2+(4a+2)x+1$ of $h_p(x)$ correspond to the same number of factors of $P_{\frac{p-1}{4}}^{(-1/4,0)}(x)$ of
the form $x^2 - (6 + 8a)x + 9 + 8a$, or equivalently, to factors $x^2+ax+b$ for which $a+b \equiv 3$ (mod $p$).  See \cite[pp. 80-81]{BM04} for
an analogous result for the Legendre polynomial $P_{(p-1)/2}(x)$.

\subsection{Primes $p \equiv 3$ mod $4$.}

\begin{lem}[AI]
If $p \equiv 3$ (mod $4$) and $x^2+a$ is a binomial quadratic factor of $P_n(x)$ (mod $p$), then both $a$ and $a+1$ are nonzero squares in
$\F_p$.
\label{lem:3}
\end{lem}

\begin{proof}
Irreducibility of $x^2+a$ says that $-a$ is a nonsquare in $\mathbb{F}_p$.  In this case $-1$ is a nonsquare, 
so $a$ is a nonzero square.  That $a+1$ is a square follows as in the proof of Theorem \ref{thm:9} using \eqref{eqn:3.12}.
\end{proof}

For the next three results we follow AI and let
$$\chi_p(a) = \left(\frac{a}{p}\right) \equiv a^{(p-1)/2}, \ a \in \mathbb{F}_p^\times,$$
denote the Legendre symbol;
$$\mathcal{R}_p^{\pm} = \{\beta \in \mathbb{F}_p^\times: \chi_p(\beta) = \pm 1, h_p(\beta) = 0\};$$
and
$$\mathcal{B}_p = \{x^2+a \ \textrm{irreducible} \in \mathbb{F}_p[x]: x^2 + a \mid P_n(x)\}.$$

\begin{prop}[AI]
Taking reciprocals preserves $\mathcal{R}_p^{-}$ and there is a bijection
$$\left(\mathcal{R}_p^{-} \setminus \{-1\}\right)/(\beta \sim \beta^{-1}) \ \longleftrightarrow \ \mathcal B_p.$$
Furthermore, $-1\in \mathcal{R}_p^-$ if and only if $p \equiv 7$ (mod $8$).
\label{prop:4}
\end{prop}

\begin{proof}
The fact that $h_p(x)$ is a reciprocal polynomial shows that $\mathcal{R}_p^{-}$ is invariant under
taking reciprocals.  Furthermore, a reciprocal polynomial of degree $n$ with no multiple roots has $-1$ as a root 
if and only if $n$ is odd.  Hence, by \eqref{eqn:3.3}, 
$h_p(-1) = 0$ if and only if $n = \frac{p-3}{4}$ is odd, i.e., $p \equiv 7$ (mod $8$). \medskip
 
 Let $\beta \in \mathcal{R}_p^-\setminus \{-1\}$ and choose $u \in \mathbb{F}_{p^2}^\times$ with $u^2=\beta$.  Since $\beta$
is a nonsquare, $u^p=-u$.  If $z=\frac{u+u^{-1}}{2}$, then $z^p=-z$, and by Proposition \ref{prop:1}, $P_n(z)=0$.  Since
$\beta \neq- 1$, $z \neq 0$.  Hence, the minimal polynomial of $z$ over $\F_p$,
$$x^2-z^2 = x^2+a, \ \ a= -z^2 = -\frac{(\beta+1)^2}{4\beta},$$
is an element of $\mathcal{B}_p$.  This minimal polynomial is invariant under
$\beta \mapsto \beta^{-1}$, and the equation recovering $\beta$ from $a$ is
\begin{equation}
\beta^2+(4a+2)\beta+1=0,
\label{eqn:3.13}
\end{equation}
as in \eqref{eqn:3.11}.  Thus a factor in $\mathcal{B}_p$ can arise from at most one reciprocal pair. \medskip

Conversely, take a factor $x^2+a \in \mathcal{B}_p$.  By Lemma \ref{lem:3}, choose $s,t \in \mathbb{F}_p^\times$ with
$s^2=a$ and $t^2 = a+1$.  Let $\iota \in \mathbb{F}_{p^2}$ satisfy $\iota^2=-1$ and set
$$\beta_+=-(t+s)^2,\ \ \beta_-=-(t-s)^2.$$
This pair is independent of the choices of signs of $s$ and $t$, up to interchanging them.  Since $(t+s)(t-s) = (a+1)-a = 1$, one has
$\beta_+\beta_-=1$ and $\beta_{+}+\beta_-=-(4a+2)$, so $\beta_\pm$ are precisely the two roots of
\eqref{eqn:3.13}.  They are distinct, since equality would force $\beta_\pm=-1$ and then $st=0$, contrary to $a \neq 0, -1$. \medskip

Now set $u=\iota(t+s)$.  Then $u^2=\beta_+$, $u^{-1}=-\iota(t-s)$, and
$$ \frac{u+u^{-1}}2=\iota s, \ \  (\iota s)^2=-a.$$
Thus $\pm \iota s$ are roots of $x^2+a$ and hence of $P_n(x)$.  Equation \eqref{eqn:3.4} gives $h_p(\beta_{\pm})=0$.  
Moreover $\iota^p=-\iota$, so $u^p=-u$ and $\beta_+^{(p-1)/2} = u^{p-1} = -1$; so both values are nonsquares.  This constructs
a reciprocal root pair in $\mathcal{R}_p^-$ for every factor $x^2+a \in \mathcal{B}_p$.  
By \eqref{eqn:3.11}, this is inverse to the forward construction and proves the bijection.
\end{proof}

\begin{thm}[AI]
If the prime $p \equiv 3$ (mod $4$),
$$\#\mathcal{R}^{-}_p = \#\{\beta\in\F_p^\times:\chi_p(\beta)=-1, \ h_p(\beta)=0\} = \frac{h(-2p)-2}{2},$$
where $h(-2p)$ is the class number of $\mathbb{Q}(\sqrt{-2p})$.
\label{thm:10}
\end{thm}

\begin{proof}
By Theorem 1.1 of \cite{Mort10}, we have
$$\#\mathcal{B}_p=\frac{h(-2p)-d_p}{4}, \ \ d_p=\begin{cases}
        2, &p \equiv 3 \ (\textrm{mod} \ 8),\\
        4, &p \equiv 7 \ (\textrm{mod} \ 8).
   \end{cases}$$
If $p \equiv 3$ (mod $8$), $\mathcal{R}_p^-$ has no exceptional fixed root under $\beta \sim \beta^{-1}$ and therefore 
$\#\mathcal{R}_p^-= 2\#\mathcal{B}_p$.  If $p \equiv7$ (mod $8$), Proposition \ref{prop:4} gives 
$\#\mathcal{R}_p^-= 2\#\mathcal B_p+1$, the extra root being $-1$.  Substitution yields $\#\mathcal{R}_p^- = \frac{h(-2p)-2}{2}$ in both cases.
\end{proof}

\begin{thm}[AI]
Let $p > 3$ be a prime with $p \equiv 3$ (mod $4$), and put
$$\mathcal{R}_p^\pm:=\{\beta \in \F_p^\times:h_p(\beta)=0,\ \chi_p(\beta)=\pm1\}.$$
Then
$$\#\mathcal{R}_p^+= \begin{cases}
   3h(-p)-1,&p \equiv 3 \ \textrm{mod} \ 8,\\
   h(-p)-1,  &p \equiv 7 \ \textrm{mod} \ 8,
 \end{cases}$$
 and
$$ \#\mathcal{R}_p^-=\frac{h(-2p)-2}{2}.$$
Consequently,
$$\#\{\beta \in \mathbb{F}_p: h_p(\beta)=0\} = \begin{cases}
   3h(-p)+\frac{h(-2p)}{2}-2, &p \equiv 3 \ \textrm{mod} \ 8,\\
   h(-p)+\frac{h(-2p)}{2} - 2, &p \equiv 7 \ \textrm{mod} \ 8.
 \end{cases}
$$
\label{thm:11}
\end{thm}

\begin{proof}(AI, M)
The count for $\mathcal{R}_p^-$ is Theorem \ref{thm:10}.  The same argument as in the proof of Theorem \ref{thm:8}, using \eqref{eqn:3.5}, shows
 that there is a 1-1 correspondence between the roots of $h_p(x)$ in $\mathbb{F}_p$ which are squares and the roots of $W_{(p-3)/4}(x) = W_n(x)$
which are squares in $\mathbb{F}_p$.  Then Theorem \ref{thm:6} (also see \cite[Thm. 1(c)]{BM04}) yields the count for $\mathcal{R}_p^+$. 
\end{proof}

\noindent {\bf Remark.} The proof originally found by AI relied on a theorem of Ando \cite[Thm. A]{And25}.  The proof given here only relies on the earlier
 paper of Brillhart and Morton \cite{BM04}.  This also shows that Ando's Theorem A can be derived from the results of that paper. \medskip

\begin{thm}[M]
If $p \equiv 3$ (mod $4$), then the number $N_2(h,p)$ of irreducible quadratic factors of the form $x^2 + bx +1$ of $h_p(x)$ (mod $p$) is given by
$$N_2(h,p) = \begin{cases} 0, & p \equiv 3 \ \textrm{mod} \ 8;\\
\frac{h(-p)-1}{2}, & p \equiv 7 \ \textrm{mod} \ 8. \end{cases} $$
\label{thm:12}
\end{thm}

\begin{proof}
We appeal to the formula \eqref{eqn:3.5} connecting $h_p(x)$ and $W_{(p-e)/4}(x) = W_n(x)$.  First we show that for $a \neq -1$, 
a linear factor $x-a$ of  $W_n(x)$ contributes an 
irreducible factor of the form $x^2+bx+1$ to the factorization of $h_p(x)$ if and only if $a$ is not a square in $\mathbb{F}_p$.  To see this, note that
\begin{align*}
x - a \rightarrow \ &(u+1)^2\big[\left(\frac{u-1}{u+1}\right)^2-a\big] = (u-1)^2-a(u+1)^2\\
 = \ &(1-a)u^2-2(a+1)u+(1-a) = (1-a)\left(u^2+2\frac{a+1}{a-1} u +1\right).
\end{align*}
Replacing $a$ by $1/a$ yields the factor $u^2-2\frac{a+1}{a-1} u +1$, so the product of these two factors,
$$\left(u^2+2\frac{a+1}{a-1} u +1\right)\left(u^2-2\frac{a+1}{a-1} u +1\right) = u^4 - 2\frac{a^2 + 6a + 1}{(a-1)^2} u^2 + 1,$$
divides $h_p(u^2)$ and contributes the factor
$$t(x) = x^2 - 2\frac{a^2 + 6a + 1}{(a-1)^2} x + 1, \ \ \textrm{with} \ \ \textrm{disc}(t(x)) = \frac{64a(1 + a)^2}{(a-1)^4},$$
to $h_p(x)$.  If $\chi_p(a) = +1$, then $t(x)$ splits into linear factors, while if $\chi_p(a) = -1$, $t(x)$ is irreducible over $\mathbb{F}_p$. \medskip

Next, we show that no irreducible quadratic factor of $W_n(x)$ contributes a quadratic factor of the desired form.  Given an irreducible factor $f(x) = x^2+ax+b$, 
for which, in particular, $b \neq -a-1$, we 
have
\begin{align*}
&T(u) = (u+1)^4 \left( \left(\frac{u-1}{u+1}\right)^4+a \left(\frac{u-1}{u+1}\right)^2 + b \right)\\
&= (a + b + 1)u^4 + (4b - 4)u^3 + (-2a + 6b + 6)u^2 + (4b - 4)u + a + b + 1.
\end{align*}
Now if an irreducible $x^2+rx+1$ divides $h_p(x)$, then $q(u) = u^4+ru^2+1$ divides $h_p(u^2)$.  Furthermore, $\textrm{disc}(q(u)) = 16(r^2-4)^2$, so by the
Pellet-Stickelberger-Voronoi Theorem \cite[Appendix]{BM04}, $q(u)$ factors modulo $p$ into an even number of distinct factors.  It cannot be divisible by a linear polynomial, 
since $x^2+rx+1$ has no roots in $\mathbb{F}_p$.  It also cannot be a product $(u^2+v)(u^2+1/v)$, since this would imply $x^2+rx+1$ is reducible. Hence, it must factor as
$$u^4+ru^2+1 = (u^2+cu+t)(u^2-cu+t), \ \ c, t \in \mathbb{F}_p, \ t^2 = 1.$$
It follows that $f(x)$ can contribute to a factor of the form $x^2+rx+1$ of $h_p(x)$ if and only if the above polynomial $T(u)$ is divisible by $u^2+cu+t$, where $t=\pm 1$.  
Taking the remainder of $T(u)$ by $u^2+cu+1$ yields the linear polynomial
$$c(-ac^2 - bc^2 + 4bc - c^2 + 4a - 4b - 4c - 4)u - ac^2 - bc^2 + 4bc - c^2 + 4a - 4b - 4c - 4;$$
this is $0$ if and only if the constant term is $0$, i.e., if and only if
$$(a + b + 1)c^2 + (-4b + 4)c - 4a + 4b + 4 = 0.$$
But this quadratic has discriminant equal to $16(a^2-4b) \notin (\mathbb{F}_p^\times)^2$, so it has no roots in $\mathbb{F}_p$.  On the other hand, dividing $T(u)$ by
$u^2+cu-1$ yields the remainder
$$(-ac^3 - bc^3 + 4bc^2 - c^3 - 8bc - 4c^2 + 8b - 8c - 8)u + ac^2 + bc^2 - 4bc + c^2 + 8b + 4c + 8.$$
If this is $0$, then the resultant
\begin{align*}
\textrm{Res}_c&(-ac^3 - bc^3 + 4bc^2 - c^3 - 8bc - 4c^2 + 8b - 8c - 8,\\
& \ \ ac^2 + bc^2 - 4bc + c^2 + 8b + 4c + 8)\\
&= 64(a+b+1)^3(b-1)^2
\end{align*}
must be $0$, giving $b = -1-a$ or $b = 1$.  The former gives a reducible polynomial $x^2+ax-1-a = (x-1)(x+1+a)$.  Assume that $b = 1$.  Then the fact that
both terms in the resultant are $0$ yields that
\begin{equation}
(a + 2)c^2 + 16 = 0, \ i.e., \ c^2 = -\frac{16}{a+2}.
\label{eqn:3.14}
\end{equation}
We show that $\chi_p(a+2) = +1$ in this case, i.e., when $W_{(p-3)/4}(x)$ is divisible by $x^2+ax+1$.  For this we use the congruence 
\eqref{eqn:2.5}.  By this congruence, the factor $x^2+ax+1$ of $W_{(p-3)/4}(x)$ corresponds to the (reducible) factor
\begin{align*}
k(x) & = (4x(1-x))^2 + a(4x(1-x)) + 1\\
& = 16x^4 - 32x^3 + (-4a + 16)x^2 + 4ax + 1,\\
\textrm{disc}&(k(x)) = 2^{16}(a-2)^2(a+2)^3,
 \end{align*}
of $W_{(p-1)/2}(x)$.  This quartic cannot be divisible by a linear factor; otherwise, there would be a root $\rho = 4\lambda(1-\lambda) \in \mathbb{F}_p$
of $x^2+ax+1$.  It follows that it must factor as a product of two quadratics, and the PSV theorem implies that the discriminant of $k(x)$ is a
nonzero square (mod $p$), i.e., $\chi_p(a+2) = +1$.   Hence, there is no solution $c \in \mathbb{F}_p$ of \eqref{eqn:3.14} and no 
factor $u^2+cu-1$ dividing $T(u)$.  Therefore, no irreducible factor $f(x) = x^2 + ax + b$ contributes to a factor of $h_p(x)$ of the form $x^2+rx+1$. \medskip

Now, by Theorem \ref{thm:6}(a), all of the roots of $W_n(x)$ in $\mathbb{F}_p$ are squares if $p \equiv 3$ (mod $8$), so there can be no irreducible 
factors of the form $x^2+bx+1$ of $h_p(x)$ in this case.  Also, if $p \equiv 7$ (mod $8$), then by Theorem \ref{thm:6}(b), 
$W_n(x)$ has $h(-p)$ roots in $\mathbb{F}_p^\times - (\mathbb{F}_p^\times)^2$, one of which is $a = -1$.  This root 
contributes the factor $u^2+1$ to $h_p(u^2)$ and therefore only a linear factor $x+1$ to $h_p(x)$.  For the other nonsquare roots, $a \neq 1/a$, so a pair of  reciprocal 
nonsquare roots contributes a single factor $x^2+bx+1$ to $h_p(x)$.  Hence there are $\frac{h(-p)-1}{2}$ pairs of such roots and the same number of factors of $h_p(x)$
 of the form $x^2+bx+1$.
\end{proof}

\section{The second IKO polynomial.}

As we did in \eqref{eqn:3.3} for $h_p(x)$, we rewrite the polynomial $g_p(x)$ as follows, by setting
\begin{align*}
&p = 6n+e, \ e \in \{1,5\}, \ \frac{p-1}{2} = 3n+2s,\\
 &[p/3] = 2n+s, \ [(p+1)/6] = n+s, \ s \in \{0,1\}.
 \end{align*}
 With this notation we have the formula over $\mathbb{F}_p$:
\begin{equation}
g_p(x) \equiv  \sum_{j=0}^{2n+s}{\binom{3n+2s}{n+s+j} \binom{3n+2s}{j}x^j} = \sum_{j=0}^{2n+s}{\binom{3n+2s}{2n+s-j} \binom{3n+2s}{j}x^j}.
\label{eqn:4.1}
\end{equation}

Note that $g_p(x)$ always has linear factors over $\mathbb{F}_p$ when $p \equiv 2$ (mod $3$); namely, $x+1$ is always a factor.  This is because $g_p(x)$ 
is a reciprocal polynomial of odd degree.  Moreover, the number of linear factors is always odd in this case, for the same reason, 
because $g_p(x)$ factors into a product of linear and quadratic polynomials over $\mathbb{F}_p$. \medskip

Corresponding to Proposition \ref{prop:1}, we have

\begin{prop} If $p$ is a prime $> 3$, $n = \frac{p-e}{6}$, $e \in \{1,5\}$, and $s = [(e+1)/6]$, then we have modulo $p$ that
\begin{equation}
g_p(u^2) \equiv (-1)^s u^{2n+s} P_{2n+s}\left(\frac{u+u^{-1}}{2}\right) = (-1)^s u^{2n+s} P_{2n+s}\left(\frac{u^2+1}{2u}\right).
\label{eqn:4.2} 
\end{equation}
\label{prop:5}
\end{prop}

\noindent {\bf Remark.} It follows easily from this representation that $g_p(x)$ has no multiple roots (mod $p$), just as in the remark 
following the proof of Proposition \ref{prop:1}.

\begin{proof}
As before, we have the identity
\begin{equation}
u^{2n+s}P_{2n+s}\left(\frac{u^2+1}{2u}\right) = 4^{-2n-s} \sum_{k=0}^{2n+s}{\binom{2k}{k} \binom{4n+2s-2k}{2n+s-k}u^{2k}} \ \  \textrm{in} \ \mathbb{Q}[u].
\label{eqn:4.3}
\end{equation}
Modulo $p=6n+e$ we have $3n + 2s \equiv \frac{p-1}{2} \equiv -\frac{1}{2}$ (mod $p$), hence
\begin{align*}
\binom{3n + 2s}{k} &\equiv (-1)^k \frac{(\frac{1}{2})_k}{k!} \ (\textrm{mod} \ p) \ \ \textrm{and}\\
\binom{3n + 2s}{2n+s-k} &\equiv (-1)^{2n+s-k} \frac{(\frac{1}{2})_{2n+s-k}}{(2n+s-k)!} \ (\textrm{mod} \ p).
\end{align*}
Combining these with
$\binom{2r}{r}=4^{r}\frac{(\frac{1}{2})_r}{r!}$ gives
\begin{align*}
4^{-2n-s}\binom{2k}{k}\binom{4n+2s-2k}{2n+s-k} &= \frac{(\frac{1}{2})_k(\frac{1}{2})_{2n+s-k}}{k! (2n+s-k)!}\\
&\equiv(-1)^s \binom{3n+2s}{k} \binom{3n+2s}{2n+s-k} \ \textrm{mod} \ p,
\end{align*}
 for $0 \le k \le 2n+s$.  Now \eqref{eqn:4.3} and \eqref{eqn:4.1} imply that
$$u^{2n+s}P_{2n+s}\left(\frac{u^2+1}{2u}\right)  \equiv (-1)^s g_p(u^2) \ (\textrm{mod} \ p).$$
\end{proof}

\begin{cor}
The polynomial $g_p(u^2)$ has the representation
\begin{equation}
g_p(u^2) \equiv (-4)^{-2n-s} (u+1)^{4n+2s} W_{(p-\bar e)/3}\left(\left(\frac{u-1}{u+1}\right)^2\right) \ (\textrm{mod} \ p),
\label{eqn:4.4} 
\end{equation}
where $p \equiv \bar e$ (mod $3$) and $\bar e \in \{1,2\}$.
\label{cor:4}
\end{cor}

\begin{proof}
As in the proof of Corollary \ref{cor:2}, we use the fact that
$$(z+1)^nW_n\left(\frac{z-1}{z+1}\right) = 2^n P_n(z).$$
Then, with $n = (p-e)/6, s \in \{0,1\}$ as in the proposition and $z = \frac{u^2+1}{2u}$, we have $\frac{z-1}{z+1} = \left(\frac{u-1}{u+1}\right)^2$, and
\begin{align*}
(-1)^s u^{2n+s} &P_{2n+s}\left(\frac{u^2+1}{2u}\right)\\
&= (-2)^{-2n-s} u^{2n+s} \left(\frac{u^2+1}{2u}+1\right)^{2n+s} W_{2n+s}\left(\frac{z-1}{z+1}\right)\\
&= (-4)^{-2n-s} (u+1)^{4n+2s} W_{2n+s}\left(\left(\frac{u-1}{u+1}\right)^2\right) \ (\textrm{mod} \ p).
\end{align*}
Here, $2n+s = [p/3] = (p - \bar e)/3$.
\end{proof}

\noindent {\bf Remark.} Though Proposition \ref{prop:5} has the same shape as Proposition \ref{prop:1}, the two polynomials $g_p(x)$ and $h_p(x)$ 
are generally related to Legendre polynomials of different degrees, since $(p-e)/4$ is not equal to $(p - \bar e)/3$, for $e = 1,3$ and $\bar e = 1, 2$, 
except when $e=1, \bar e = 2$ and $p = 5$. In that case $g_5(x) = 2x+2 = h_5(x)$ (mod $5$).  It is interesting that the two polynomials $g_p(x)$
and $h_p(x)$, for $p >5$, arise from two different subsequences of the Legendre polynomials $P_n(x)$. \medskip

\begin{prop}
(a) For any prime $p \ge 5$, the roots of $W_{(p-\bar e)/3}(x)$ are cubes in $\mathbb{F}_{p^2}$.  In particular, the constant terms of irreducible quadratic factors
of $W_{(p-\bar e)/3}(x)$ in $\mathbb{F}_p[x]$ are cubes in $\mathbb{F}_p$. \medskip

(b) If $p \equiv 5$ (mod $12$), all $h(-p)-1$ of the roots of $W_{(p-2)/3}(x)$ in $\mathbb{F}_p$ are squares in $\mathbb{F}_{p}$. \medskip

(c) If $p \equiv 11$ (mod $12$), $h(-p)$ of the roots of $W_{(p-2)/3}(x)$ in $\mathbb{F}_p$ are nonsquares, and $3h(-p)-1$, resp., $h(-p)-1$ are squares, 
according as $p \equiv 3$ or $7$ (mod $8$).
\label{prop:6}
\end{prop}

\begin{proof}
(a) By \cite[Thm. 5]{BM04}, the roots $\alpha$ of $W_{(p-\bar e)/3}\left(1-\frac{x^3}{27}\right)$ are the nonzero supersingular parameters 
for the Deuring normal form
$$E_{3}: \ Y^2 + \alpha XY + Y = X^3.$$
By \cite[Thm. 2.3]{Mor11}, these parameters lie in the field $\mathbb{F}_{p^2}$.  Hence, the roots $\rho$ of  $W_{(p-\bar e)/3}(x)$ have the form
$$\rho = 1-\frac{\alpha^3}{27} = -\frac{1}{27} (\alpha^3-27) = \frac{-1}{27} \frac{27\alpha^3}{\beta^3} = -\frac{\alpha^3}{\beta^3},$$
where $27\alpha^3 + 27\beta^3 = \alpha^3 \beta^3$ and $\beta$ is also a root of $W_{(p-\bar e)/3}\left(1-\frac{x^3}{27}\right)$ (see \cite[Thm. 1.3(c)]{Mor11}).
Hence $\rho = (-\alpha/\beta)^3 = \gamma^3$ is a cube in $\mathbb{F}_{p^2}$.  
For the second assertion, the constant term of an irreducible quadratic factor of $W_{(p-\bar e)/3}(x)$
 has the form $\gamma^3 \gamma^{3p} = N(\gamma)^3$, where $N$ denotes the norm to $\mathbb{F}_p$. \medskip
 
 \noindent (b) For this part we consider the points $(0,0), (0,-1)$ of order $3$ on the elliptic curve $E_{3}$.
The doubling formula for $X$-coordinates on this curve is
 $$X(2P) = \frac{x(x^3-\alpha x -2)}{4x^3+(\alpha x+1)^2}, \ \ x = X(P).$$
See \cite[p. 54]{Si09}.  Assume $p \equiv 5$ (mod $12$).  Let $\rho = 1-\frac{\alpha^3}{27} \in \mathbb{F}_p$ 
be a root of $W_{(p-2)/3}(x)$, where $\alpha \in \mathbb{F}_p$ is a supersingular parameter for $E_{3}$ over $\mathbb{F}_p$.  Then
 $$|E_{3}(\mathbb{F}_p)| = p + 1 \equiv 0 \ (\textrm{mod} \ 6).$$
 It follows that there is a point $P \in E_{3}(\mathbb{F}_p)$ for which $2P = (0,0)$ and $P = (b,c)$, with $ b \neq 0$ in $\mathbb{F}_p$.  The
 formula for $X(2P)$ implies that $b^3-\alpha b -2 = 0$, so that $\alpha = (b^3-2)/b$.  Then $(b,c)$ on $E_3(\alpha)$ implies that
 $$c^2 + \alpha b c + c - b^3 = c^2 + (b^3-1)c - b^3 = (c-1)(c+b^3) = 0,$$
so that $c \in \mathbb{F}_p$.  This shows that $c \in \mathbb{F}_p(b)$,
 for any root $b$ of $x^3-\alpha x -2 = 0$.  Now $p+1 \equiv 2$ (mod 4), so there can be only one
 point of order $2$ in $E_{3}(\mathbb{F}_p)$.  It follows that the other two roots of $x^3-\alpha x -2$ cannot lie in $\mathbb{F}_p$, so must 
 satisfy an irreducible quadratic.  If $D =\textrm{disc}(x^3-\alpha x -2) = 4(\alpha^3-27)$, the PSV theorem \cite[Appendix]{BM04} gives that
 $$\left(\frac{D}{p}\right) = (-1)^{3+r},  \ \ r = \#\{\textrm{irred. factors of} \ x^3-\alpha x -2 \ \textrm{over} \ \mathbb{F}_p\},$$
hence that $D$ is not a square in $\mathbb{F}_p$.  But $4(\alpha^3-27) = -108(1-\alpha^3/27) = -108\rho$ implies that
 $$\left(\frac{\rho}{p}\right) = \left(\frac{-3}{p}\right) \left(\frac{-108\rho}{p}\right) =  \left(\frac{p}{3}\right) (-1) = +1.$$
 Hence $\rho$ is a square in $\mathbb{F}_p$. \medskip
 
 \noindent (c) The above argument shows that a root $\rho \in \mathbb{F}_p$ of $W_{(p-2)/3}(x)$ is a nonsquare if and only if 
 $D = 4(\alpha^3-27)$ is a square in $\mathbb{F}_p$.  If $D$ is a square, then $x^3-\alpha x-2$ must have an odd number of irreducible factors over
 $\mathbb{F}_p$, hence must split completely, since it has at least one root in $\mathbb{F}_p$.  Thus, we must show that
 \begin{equation*}
 \#\{\alpha \in \mathbb{F}_p: W_{(p-2)/3}\left(1-\frac{\alpha^3}{27}\right) = 0, x^3-\alpha x-2 \ \textrm{splits}\} = h(-p).
 \end{equation*} 
 Now we write the $j$-invariant of $E = E_{3}$ in terms of $\rho= 1-\frac{\alpha^3}{27}$, using $\alpha^3 = 27(1-\rho)$:
 \begin{equation}
 j(E) = \frac{\alpha^3(\alpha^3-24)^3}{\alpha^3-27} = \frac{27(1 - \rho)(9\rho-1)^3}{\rho}.
 \label{eqn:4.5}
 \end{equation}
 With this we factor the quantity $j-1728$:
 \begin {equation}
 j(E)-1728 = -\frac{27(27\rho^2 - 18\rho - 1)^2}{\rho}.
 \label{eqn:4.6}
 \end{equation}
 If $\rho \in \mathbb{F}_p$ and $\chi_p(\rho) = -1$, then since $\left(\frac{-3}{p}\right) = -1$, we have $j(E) - 1728 \in (\mathbb{F}_p^\times)^2$.  
 Conversely, assume $j = j(E) \neq 0, 1728$ is a supersingular $j$-invariant in $\mathbb{F}_p$ and $\chi_p(j-1728) = +1$.  Then the discriminant of
 \begin{align*}
 F(\rho,j) &= -27(1 - \rho)(9\rho-1)^3+j \rho\\
 &= 3^9 \rho^4 - 26244\rho^3 + 7290\rho^2 + (-756 + j)\rho + 27
 \end{align*}
 is
 $$d = \textrm{disc}(F(\rho,j)) = -3^{21}j^2(j-1728)^2 \ \notin \ (\mathbb{F}_p^\times)^2.$$
 Using the PSV Theorem again, this implies that $F(\rho,j)$ must have an odd number of irreducible factors over $\mathbb{F}_p$.  It cannot be
 irreducible; otherwise $W_{(p-2)/3}(x)$ would have an irreducible quartic factor, which is not the case, by \cite[Prop. 4]{BM04}.  
 Hence, it must factor as a product of two linear polynomials and an irreducible quadratic.  
 This says that there is a pair of roots $\rho$ of $W_{(p-2)/3}(x)$ in $\mathbb{F}_p$ for which
 \eqref{eqn:4.5} and \eqref{eqn:4.6} hold, and for which $\chi_p(\rho) = -1$. \medskip
 
 Now we appeal to \cite[Props. 8, 10(b)]{BM04}.  These results say that $j(E)-1728$ is a square (when $p \equiv 3$ mod $4$) if and only if 
 $\mu = \frac{1+\sqrt{-p}}{2}$ injects into the endomorphism ring $\textrm{End}(E)$, and that there are exactly $\frac{h(-p)+1}{2}$ such supersingular 
 $j$-invariants (including $j=1728$).  Hence, there are $2$ times $\frac{h(-p)-1}{2}$ roots $\rho$ of $W_{(p -2)/3}(x)$ corresponding to
 these supersingular $j$-invariants, other than $j = 1728$.  For $j=1728$, the corresponding $\rho$ values are given by
 $$27\rho^2 - 18\rho - 1 \equiv 0 \ (\textrm{mod} \ p), \ \textrm{disc}(27\rho^2 - 18\rho - 1) = 432 = 2^4 3^3.$$
 Since $\left(\frac{3}{p}\right) = +1$, there are two roots $\rho \in \mathbb{F}_p$ of $W_{(p-2)/3}(x)$ corresponding to $j = 1728$.  However, their
 product is $-1/27$, so one is a quadratic residue and one is a nonresidue (mod $p$).  Hence the total number of roots $\rho \in \mathbb{F}_p$ with
 $\chi_p(\rho) = -1$ is
 $$2 \cdot \frac{h(-p)-1}{2} + 1 = h(-p).$$
 This proves the first statement in (c), and the second follows from \cite[Thm. 1(d)]{BM04}.
 \end{proof}

\begin{lem} If $E = E_{3}(\alpha)$ is supersingular in characteristic $p \ge 5$, with $\mathbb{F}_p(\alpha) = \mathbb{F}_{p^2}$,
 then the Frobenius map $\pi :(x,y) \rightarrow (x^{p^2},y^{p^2})$ is $\pi = \left(\frac{p}{3}\right)p$.  Moreover, 
 $E(\mathbb{F}_{p^2}) = E[p - \left(\frac{p}{3}\right)] \cong \mathbb{Z}_{p \pm 1} \oplus
\mathbb{Z}_{p \pm 1}$.
\label{lem:4}
\end{lem}

\begin{proof}
First assume $p \equiv 1$ (mod $3$).  For this case we 
use the same argument as in the first part of the proof of \cite[Thm. 6.1(i)]{Mor06} applied to the curve
$$E_{3}(\alpha): Y^2 + \alpha XY + Y = X^3,$$
instead of the curve
$$E_n: Y^2 + aXY + bY = X^3 + bX^2.$$
In that argument we take the meromorphism $\mu = p$ and $\ell = p^2$, and use
$$\alpha^\ell = A,  \ x^\mu = cx^{\ell} =cX, \ y^\mu = d y^{\ell} = dY$$
to compare the equation obtained by raising everything to the power $\ell = p^2$, $Y^2+AXY+Y = X^3$, with the equation obtained by applying $\mu$:
\begin{align*}
&(y^\mu)^2 + \alpha x^\mu y^\mu + y^\mu = (x^\mu)^3, \ \ \textrm{or}\\
&d^2Y^2+\alpha cd XY + dY = c^3X^3.
\end{align*}
This leads to the equations 
\begin{equation*}
d^2A = \alpha cd, \ \ d^2 = d, \ \ d^2 = c^3.
\end{equation*}
This gives $d=1$ and $c=1$, since we know that $A = \alpha^{p^2} = \alpha$.  Hence, $\mu$ satisfies $(x,y)^\mu = (x^\ell, y^\ell) 
= (x^q, y^q) = (x,y)^\pi$.  Thus we get the conclusion $\pi = \mu = p$.  \medskip

Now let $p \equiv 2$ (mod $3$).  We appeal to the second part of the proof of \cite[Thm. 6.1(i)]{Mor06}, according to which 
$$x^\mu = cx^\ell = cX, \ \ y^\mu = d(y + \alpha x + 1)^\ell = d(Y + A X + 1).$$ 
This uses that
$$(x) = \frac{{\bf p}_1 {\bf p}_{-1}}{{\bf o}^2}, \ (y) =  \frac{{\bf p}_1^3}{{\bf o}^3}, \ (y+\alpha x + 1) =  \frac{{\bf p}_{-1}^3}{{\bf o}^3},$$
noting that the prime divisors ${\bf p}_1,  {\bf p}_{-1}$ of $\textsf{K} = \overline{\mathbb{F}}_p(x,y)$ correspond to the points $(0,0)$ and $(0,-1)$ of order $3$.  
Now comparing the equations
\begin{align*}
d^2Y^2+d^2AXY&+d^2Y = d^2X^3,\\
d^2(Y+AX+1)^2+\alpha cdX(Y&+AX+1)+d(Y+AX+1) = c^3X^3,
\end{align*}
yields the equations:
\begin{align*}
\textrm{coeff. of} \ XY:& \ 2d^2A + \alpha cd = d^2A,\\
\textrm{coeff. of} \ Y:& \ 2d^2 + d = d^2.
\end{align*}
The second gives that $d = -1$, and the first gives that $A = -\alpha cd = \alpha c$, whence $c = 1$ since $A = \alpha^\ell = \alpha$.  It follows that
$$(x,y)^\mu = (x^\ell, -(y+\alpha x +1)^\ell) = -(x^\ell, y^\ell) = -(x^q,y^q) = (x,y)^{-\pi}.$$
Thus, $\pi = -\mu = -p$. \medskip

Finally, since $\pi = \pm p$, $\pi(P) = P$ for every point $P \in E[p \mp 1]$, we see that $E[p \mp 1] \subseteq E(\mathbb{F}_{p^2})$.
Comparing orders yields that $|E[p \mp 1]| = (p \mp 1)^2 = |E(\mathbb{F}_{p^2})|$, since $\pi^2 -a \pi +p^2 = 0$ with $a = \pm 2p$ implies that
$$|E(\mathbb{F}_{p^2})| = p^2+1 - a = (p \mp 1)^2.$$
See \cite[pp. 130, 141]{Wa08}.  This gives the isomorphism of the lemma.
\end{proof}

\begin{prop}
The roots of $W_{(p- \bar e)/3}(x)$ in characteristic $p$ are squares in $\mathbb{F}_{p^2}$, and the same holds for the roots of $g_p(x)$.
\label{prop:7}
\end{prop}

\begin{proof}
Let $\alpha$ be a root of $W_{(p- \bar e)/3}(x)$.  The first assertion is obvious if $\alpha \in \mathbb{F}_p$, so we may assume $\mathbb{F}_p(\alpha) = \mathbb{F}_{p^2}$.  By Lemma \ref{lem:4}, there are three points of order $2$ in $E_3(\mathbb{F}_{p^2})$, and hence three distinct points $Q \neq (0,-1)$ on this curve for which 
$2Q = P = (0,0)$.  By the doubling formula on $E_3$ used in the proof of Proposition \ref{prop:6}, this implies that there are three distinct solutions of
$x^3-\alpha x -2 = 0$ in $\mathbb{F}_{p^2}$.  Hence, its discriminant must be a square in this field: $4(\alpha^3-27) = \prod_{i<j}{(\rho_i-\rho_j)^2}$.  On the other hand,
the roots of $W_{(p-\bar e)/3}(x)$ (none of which are $0$) have the form $\rho = 1-\frac{\alpha^3}{27} = \frac{-1}{27}(\alpha^3-27)$, which is a square in $\mathbb{F}_{p^2}$.
Now \eqref{eqn:4.4} implies that the roots of $g_p(x)$ are also squares in $\mathbb{F}_{p^2}$, as in the proof of Corollary \ref{cor:3} to Proposition \ref{prop:1}.
\end{proof}

Corresponding to Proposition \ref{prop:3} we have

\begin{prop}
If $p \equiv \bar e$ (mod $3$), $\bar e \in \{1,2\}$, $2n +s = \frac{p-\bar e}{3}$, we have the congruences
\begin{equation*}
g_p(x) \equiv  \begin{cases} P_{2n}^{(-1/6,0)}(1-2x) \ (\textrm{mod} \ p), & p \equiv 1 \ (\textrm{mod} \ 3);\\
P_{2n+1}^{(1/6,0)}(1-2x) \ (\textrm{mod} \ p), & p \equiv 2 \ (\textrm{mod} \ 3).\end{cases}
\end{equation*}
\label{prop:8}
\end{prop}

\begin{proof}
First assume $p \equiv 1$ (mod $3$).  From \cite[p. 139]{IKO86} we know that $g_p(x) \equiv c F_{2,1}(1/3, 1/2, 5/6; x)$ 
modulo $p$, where $c = \binom{3n}{n}$.  Also,
$$P_{2n}^{(-1/6,0)}(1-2x) \equiv \frac{(5/6)_{2n}}{(2n)!}F_{2,1}(1/3, 1/2, 5/6; x) \ (\textrm{mod} \ p),$$
so we just need to check leading coefficients.  These leading coefficients are
\begin{align*}
\textrm{l.c.} \ g_p(x) &= \binom{3n}{2n}, \ \ n = \frac{p-1}{6};\\
\textrm{l.c.} \ P_{2n}^{(-1/6,0)}(1-2x) &= (-2)^{2n} 2^{-2n} \sum_{k=0}^{2n}{\binom{2n-1/6}{k} \binom{2n}{2n-k}}\\
&= \binom{4n-1/6}{2n} \equiv \binom{-5/6}{2n} \ (\textrm{mod} \ p).
\end{align*}
Thus, we must show that $\binom{3n}{2n} \equiv \binom{-5/6}{2n} \ (\textrm{mod} \ p)$.
Since $3n \equiv -1/2$ (mod $p$), this is equivalent to
$$\left(\frac{1}{2}\right)_{2n} \equiv \left(\frac{5}{6}\right)_{2n} \ (\textrm{mod} \ p);$$
and this follows from
\begin{align*}
3^{2n}(1 \cdot 3 \cdot 5 \cdots (4n-1)) &= 3 \cdot 9 \cdot 15 \cdots (12n-3), \\
5 \cdot 11 \cdots (12n-1) &\equiv (2p-5)(2p-11) \cdots (2p-12n+1)\\
&\equiv (12n-3)(12n-9) \cdots (3) \ (\textrm{mod} \ p).
\end{align*}
A similar argument proves the congruence when $p \equiv 2$ (mod $3$), using that 
\begin{align*}
g_p(x) &\equiv cF_{2,1}(2/3,1/2, 7/6; x), \ c = \binom{3n+2}{2n+1},\\
P_{2n+1}^{(1/6,0)}(1-2x) &\equiv \frac{(7/6)_{2n+1}}{(2n+1)!} F_{2,1}(2/3,1/2, 7/6; x),\\
\left(\frac{1}{2}\right)_{2n+1} &\equiv - \left(\frac{7}{6}\right)_{2n+1} \ (\textrm{mod} \ p), \ p = 6n+5.
\end{align*}
\end{proof}

\subsection{Primes $p \equiv 1$ modulo $3$}

We turn first to the primes $p = 6n+1$, where $e = 1, s = 0$ and $n = \frac{p-1}{6}$.

\begin{prop} Assume $p \equiv 1$ (mod $3$).  \begin{enumerate}[(a)]
\item If $\alpha \in \mathbb{F}_p$ is a root of $g_p(x)$, then $\alpha$ is not a square in $\mathbb{F}_p$.

\item We have that
 \begin{align*}
& \#\{\alpha \in \mathbb{F}_p:  g_p(\alpha) = 0\}\\
 & = 2\#\{b \in \mathbb{F}_p: x^2+b \mid P_{(p-1)/3}(x), \chi_p(b+1) = \chi_p(b) = (-1)^{(p+1)/2}\}.
 \end{align*}
 
 \item Each irreducible factor $x^2+b$ of $P_{(p-1)/3}(x)$ in (b) corresponds to a unique pair of elements 
 $\{\alpha, \beta\} \subset \mathbb{F}_{p^2} - \mathbb{F}_p$ satisfying 
 $$Fer_3: \ 27\alpha^3+27\beta^3 = \alpha^3 \beta^3, \ \ \beta = \alpha^p.$$
Then $\chi_p(b+1) = \psi_p(\alpha)$, where $\psi_p$ is the quadratic character on  $\mathbb{F}_{p^2}^\times$.
Hence, $b+1$ is a square in $\mathbb{F}_p$ if and only if $\alpha$ is a square in $\mathbb{F}_{p^2}$.
  \end{enumerate}
 \label{prop:9}
 \end{prop}
 
 \begin{proof} (a) Suppose $g_p(\alpha) = 0$, with $\alpha \in \mathbb{F}_p$.  We appeal to Corollary \ref{cor:4}, with $s = 0, \bar e = 1$.
  Then $u^2-\alpha$ is a factor of $g_p(u^2)$, so must also be a factor of the right side of \eqref{eqn:4.4}.  If $u^2-\alpha = (u-\rho)(u+\rho)$
   split into linear factors over $\mathbb{F}_p$, then
 $$(u+1)^{4n}W_{(p-1)/3}\left(\left(\frac{u-1}{u+1}\right)^2\right)$$
 would have linear factors, meaning that $W_{(p-1)/3}(x)$ would have a root $\gamma = \left(\frac{\rho-1}{\rho+1}\right)^2 \in \mathbb{F}_p$.  This is
 impossible, since by \cite[Thm. 1(d)]{BM04}, $W_{(p-1)/3}(x)$ has no linear factors (mod $p$).  Hence, $\alpha \notin (\mathbb{F}_p^\times)^2$.
 \medskip
 
 \noindent (b) For this part, we use Proposition \ref{prop:5}.  We first claim that the number of binomial quadratic factors of $g_p(u^2)$ is equal to 
 the number of linear factors of $g_p(x)$.  By part (a), each linear factor $x-\alpha$ of $g_p(x)$ over $\mathbb{F}_p$ contributes the 
 binomial quadratic factor $u^2-\alpha$ to $g_p(u^2)$.  On the other hand, an irreducible factor $x^2+ax+b$ could only contribute a binomial quadratic factor
  to $g_p(u^2)$ if $u^4+au^2+b$ is divisible by $u^2+c$, for some $c \in \mathbb{F}_p$.  Dividing $u^4+au^2+b$ by $u^2+c$ leaves the remainder
  $c^2-ac+b$, which is only $0$ when $c = \frac{a \pm \sqrt{a^2-4b}}{2}$, and these values do not lie in $\mathbb{F}_p$.  This proves our claim. \medskip
  
 Now an irreducible factor $x^2+ax+b$ of $P_{(p-1)/3}(x)$ contributes a binomial quadratic factor to the right side of \eqref{eqn:4.2} if and only if 
 $$(u^2+1)^2+2a u (u^2+1) +4bu^2 = u^4 + 2au^3 +(2+4b)u^2 +2a u +1$$
 is divisible by $u^2+c$, for some $c \in \mathbb{F}_p^\times$.  This holds if and only if the remainder, on dividing the former by the latter, satisfies
 $$(-2ac+2a)u-4bc+c^2-2c+1 \equiv 0 \ \textrm{mod} \ p.$$
 This gives that $a = 0$ or $c=1, b = 0$.  The second solution is impossible; and if $a = 0$, then $x^2+b$ is a binomial 
quadratic factor of $P_{(p-1)/3}(x)$.  In the latter case we have
$$b = \frac{(c-1)^2}{4c}, \ \textrm{or} \ c^2 -(2+4b)c+1 = 0,$$
which has either $0$ or $2$ solutions $c$ in $\mathbb{F}_p$. Furthermore, since $u^2+c$ must be irreducible in $\mathbb{F}_p[u]$, $-c$
 must be a nonsquare in $\mathbb{F}_p$. Then we have, since $b+1 = (c+1)^2/4c$, that for a solution to exist,
 \begin{align*}
 \chi_p(c) &= \chi_p(b) = \chi_p(b+1) = -1, \ p \equiv 1 \ (\textrm{mod} \ 4),\\
 \chi_p(c) &= \chi_p(b) = \chi_p(b+1) = +1, \ p \equiv 3 \ (\textrm{mod} \ 4).
\end{align*}
These values always hold for $\chi_p(b)$, since $x^2+b$ is an irreducible factor of $P_{(p-1)/3}(x)$.  If they also hold for $\chi_p(b+1)$, then there 
must be two solutions for $c \in \mathbb{F}_p$, since solving the above quadratic yields that $c = 1 + 2b \pm 2\sqrt{b^2 + b}$.
These arguments show that
 \begin{align*}
& \#\{\alpha \in \mathbb{F}_p:  g_p(\alpha) = 0\}\\
 & = 2\#\{b \in \mathbb{F}_p: x^2+b \mid P_{(p-1)/3}(x), \chi_p(b+1) = \chi_p(b) = (-1)^{(p+1)/2}\}.
 \end{align*}
 
 (c) From \cite[Prop. 5.2, Eq. (5.3)]{Mor11} we know that the binomial quadratic factors $x^2+b$ of $P_{(p-1)/3}(x)$ 
are in 1--1 correspondence with the irreducible quadratic factors of $W_{(p-1)/3}\left(1-\frac{t}{27}\right)$ of the form
$$t^2-\frac{108}{b+1}t+\frac{2916}{b+1} = (t-\alpha^3)(t-\beta^3),$$
whose roots $\alpha^3, \beta^3$ give a solution of $Fer_3$ in $\mathbb{F}_{p^2}$.  Thus,
$$b+1 = \frac{2916}{\alpha^3 \beta^3} = \frac{2^2 \cdot 3^6}{\alpha^{3(p+1)}}$$
and
$$\chi_p(b+1) = \chi_p(\alpha^{p+1}) \equiv \alpha^{(p^2-1)/2} = \psi_p(\alpha) \ \textrm{in} \ \mathbb{F}_{p^2},$$
where $\psi_p$ is the quadratic character on $\mathbb{F}_{p^2}^\times$.
 \end{proof}
 
 \subsection{The connection with class equations.}
 
We want to count the number of binomial quadratic factors $x^2+b$ of $P_{(p-1)/3}(x)$ 
 which give rise to a binomial quadratic factor of $g_p(u^2)$. \medskip

Set $a = \frac{-108}{b+1}$ and $u(t) = t^2+at-27a$, as in the proof of Proposition \ref{prop:9}(c) above.  We know that 
$d = \textrm{disc}(t^2+at-27a) = a(a+108)$ and $\chi_p(d) = -1$, since this quadratic is irreducible over $\mathbb{F}_p$.  \medskip

We will use the following congruence from \cite[Thm. 1.2]{Mor14}.  When $p \equiv 7$ mod $12$ and $p > 53$, the class equation 
 $H_{-12p}(X)$ satisfies
 \begin{align}
 H_{-12p}(X) \equiv H_{-8}(X)^{4\delta_2} H_{-11}(X)^{4\delta_3}H_{-35}(X)^{4\delta_6} \prod_{i}{(X^2+r_iX+s_i)^2} \ \textrm{mod} \ p,
 \label{eqn:4.7}
 \end{align}
 where
 \begin{align*}
 H_{-8}(X) &= X - 8000, \ \ H_{-11}(X) = X + 32768,\\
 H_{-35}(X) &= X^2 + 117964800X - 134217728000
 \end{align*}
 and the exponents are defined by
 \begin{align}
 \label{eqn:4.8} \delta_2 &= \frac{1}{2}\left(1-\left(\frac{-2}{p}\right)\right), \ \delta_3 = \frac{1}{2}\left(1-\left(\frac{-11}{p}\right)\right),\\
 \label{eqn:4.9} \delta_6 &= \frac{1}{4}\left(1-\left(\frac{-35}{p}\right)\right)\left(1-\left(\frac{5}{p}\right)\right).
\end{align}
Compare with \eqref{eqn:1.3} in the Introduction; in this case
$$\delta_1 = \frac{1}{2}\left(1-\left(\frac{-3}{p}\right)\right) = \delta_4 = \delta_5 = 0,$$
by \eqref{eqn:4.12}, \eqref{eqn:4.13} below.  Furthermore, each of the quadratics appearing in \eqref{eqn:4.7} is irreducible over $\mathbb{F}_p$, and satisfies 
$Q_3(r_i,s_i) = 0$, for a certain polynomial $Q_3(r,s)$ of total degree $4$ in $r,s$.  (See \cite{Mor14} or \cite[p. 269]{Mor11}.) \medskip

As shown in \cite{Mor11}, each of the factors $q(X)^e$ in this product corresponds to $\textrm{deg}(q(X))e/4$ ($=1$ or $2$) irreducible factors of the 
form $x^2+ax-27a$ of $W_{(p-1)/3}(1-x/27)$ mod $p$.  These factors arise in the following way.  
Each factor $q(X) = X^2+rX+s$ of $H_{-12p}(X)$ contributes the factor
\begin{align*}
&(t-27)^2 q\left(\frac{t(t-24)^3}{t-27}\right) = G(t,r,s),\\
&G(t,r,s) = t^8 -144t^7 +8640t^6 +(r-276480)t^5 +(4976640-99r)t^4\\
& \ \ \ + (-47775744 + 3672r)t^3 + (191102976 - 60480r + s)t^2\\
& \ \ \ + (373248r - 54s)t + 729s,
\end{align*}
to the factorization of $W_{(p-1)/3}(1-t/27)$ (see \cite[Thm. 6]{BM04}),
 which must be divisible by one or two factors of the form $u(t) = t^2+at-27a$.
Taking the remainder $A(r,s)t+B(r,s)$ of $G(t,r,s)$ by $u(t)$ and computing the 
resultant of $A(r,s)$ and $B(r,s)$ with respect to $s$ and then $r$ leads to the formulas:
\begin{equation*}
r = a^3 + 126a^2 + 2944a, \ \ s = a^4 - 576a^3 + 110592a^2 - 7077888a = a(a-192)^3.
\end{equation*}
Then
$$\textrm{disc}(q(X)) = \textrm{disc}(X^2+rX+s) = a(a+108)(a+8)^2(a+64)^2.$$

For example, when $\left(\frac{-7}{p}\right) = 1, \left(\frac{-5}{p}\right) = -1$, the quadratic $H_{-35}(X)$ contributes the two factors
\begin{align*}
H_{-35}(X): \ & \ x^2-(288 \pm 160 \sqrt{-7})x +7776 \pm 4320 \sqrt{-7},\\
& a = -(288 \pm 160 \sqrt{-7}) = -(20 \pm 4\sqrt{-7})^2, \  \ \left(\frac{a}{p}\right) = -1,
\end{align*}
where $20^2+7 \cdot 4^2 = 512$, so that $a$ is never divisible by $p$.   See \cite[pp. 271-272]{Mor11}.  \medskip
In addition, when they occur, the linear factors $H_{-8}(X) = X-8000$ 
and $H_{-11}(X) = X+32768$ in \eqref{eqn:4.7} contribute one factor each of this form:
\begin{align*}
X-8000: \ & \ x^2-8x+216, \ a = -8, \ \left(\frac{-8}{p}\right) = -1, \ \textrm{if} \ \left(\frac{-2}{p}\right) = -1;\\
X+32768: \ & \ x^2-64x+1728, \ a = -64, \ \left(\frac{-64}{p}\right) = -1,  \ \textrm{if} \ \left(\frac{-11}{p}\right) = -1.
\end{align*}
Since $a = -\frac{108}{b+1}$, we know $\chi_p(b+1) = -1$ in both of these cases. \medskip

For $p \equiv 7$ (mod $12$) we will see that the factors $X^2+r_i X+s_i$ in the above product contribute $x^2+ax-27a$, where $\chi_p(a) = -1$ 
in all cases. \medskip

\noindent {\bf Example.}
For $p = 127$, we have
$$H_{-12 \cdot 127}(X) = (X+1)^4(X + 2)^4(X^2 + 88X + 8)^4(X^2 + 17X + 38)^2 \ \textrm{mod} \ 127$$
and
\begin{align*}
a^3 + 126a^2 + 2944a - 17 & \equiv (a + 100)(a + 111)(a + 42) \ (\textrm{mod} \ 127),\\
a(a-192)^3 - 38 & \equiv (a + 119)(a + 100)(a + 40)(a + 54) \ (\textrm{mod} \ 127).
\end{align*}
Thus, $a \equiv -100 \equiv 27$ (mod $127$), and $\left(\frac{27}{127}\right) = -1$.  Then, indeed, the factor $x^2+27x+33$ divides $W_{42}(1-x/27)$
modulo $127$.  \medskip

The discussion in \cite{Mor11} shows that, in general, for the factors $X^2+rX+s$ in the product ($\Pi_i$) in \eqref{eqn:4.7},
the polynomials $a^3 + 126a^2 + 2944a-r$ and $a(a-192)^3 -s$ can only have one root $a \in \mathbb{F}_p$ in common. 
In fact, using the Euclidean algorithm, applied to these two polynomials, we find that
$$a \equiv \frac{ 702r^2 + rs + 17591910400r - 29719800s}{r^2 + 102931200r + 196100s + 262144000000000} \ \textrm{mod} \ p,$$
in terms of $r$ and $s$.  Moreover, $(r,s) = (a^3 + 126a^2 + 2944a,a(a-192)^3)$ is a parametrization of the rational curve $Q_3(r,s) = 0$ in characteristic
$0$.  The formula for $a$ in terms of $r,s$ holds as long as the denominator is not $0$.  If it is $0$, then the numerator must be $0$ also, and the common
solutions of
\begin{align*}
0 &= 702r^2 + rs + 17591910400r - 29719800s,\\
0 &= r^2 + 102931200r + 196100s + 262144000000000,
\end{align*}
are
\begin{align*}
(r,s) = & \ (-1264000, -681472000), \ (-52250000, 12167000000),\\
 & \ (117964800, -134217728000).
\end{align*}
These are exactly the coefficients of the class equations
\begin{align*}
H_{-20}(X) &= X^2-1264000X-681472000,\\
H_{-32}(X) &= X^2 -52250000X+12167000000,\\
H_{-35}(X) &= X^2 +117964800X -134217728000.
\end{align*}
These class equations are the quadratic factors which occur to the $4$-th power in the factorization \eqref{eqn:1.3}
of $H_{-3p}(X) H_{-12p}(X)$ modulo $p$.  Therefore, the formula for $a$ in terms of $r,s$ is valid for any factor 
$X^2+rX+s$ appearing in the product in \eqref{eqn:4.7}. \medskip

For $p \equiv 1$ (mod $12$) and $p >53$, we have the following factorizations modulo $p$ from \cite[Thm. 1.3]{Mor14}:
\begin{align}
\label{eqn:4.10} H_{-3p}(X) &\equiv H_{-8}(X)^{4\delta_2} H_{-20}(X)^{2\delta_4} H_{-32}(X)^{2\delta_5} \prod_{i \in I}{(X^2+r_iX+s_i)^2},\\
 \label{eqn:4.11} H_{-12p}(X) &\equiv H_{-11}(X)^{4\delta_3} H_{-20}(X)^{2\delta_4} H_{-32}(X)^{2\delta_5} H_{-35}(X)^{4\delta_6}\\
 \notag& \ \ \times \prod_{j \in J}{(X^2+r_jX+s_j)^2} \ \ \ (\textrm{mod} \ p),
 \end{align}
 where the indexing sets $I, J$ are disjoint, and in addition to \eqref{eqn:4.8}, \eqref{eqn:4.9}, we have the exponents
 \begin{align}
\label{eqn:4.12} \delta_4 &= \frac{1}{4}\left(1-\left(\frac{-5}{p}\right)\right)\left(1-\left(\frac{5}{p}\right)\right),\\
\label{eqn:4.13} \delta_5 &= \frac{1}{4}\left(1-\left(\frac{-2}{p}\right)\right)\left(1-\left(\frac{2}{p}\right)\right).
\end{align}
As above, the irreducible factors in the products in \eqref{eqn:4.10} and \eqref{eqn:4.11} have the form
\begin{align}
\label{eqn:4.14} &X^2+rX+s \equiv X^2 + (a^3 + 126a^2 + 2944a)X+a(a-192)^3 \ (\textrm{mod} \ p),\\
\notag &a \in \mathbb{F}_p, \ \chi_p(a(a+108)) = -1.
\end{align}

\begin{conj} (a) If $p \equiv 3$ (mod $4$) and $p >53$, then for each factor of the form \eqref{eqn:4.14} dividing $H_{-12p}(X)$ mod $p$, 
$\left(\frac{a}{p}\right) = -1$. \medskip

\noindent (b) If $p \equiv 1$ (mod $4$) and $p > 53$, then for each factor $q(X)$ of the form \eqref{eqn:4.14} dividing $H_{-3p}(X)$ mod $p$, 
$\left(\frac{a}{p}\right) = -1$; and for each such factor $q(X)$ dividing $H_{-12p}(X)$, $\left(\frac{a}{p}\right) = +1$.  
\label{conj:1}
\end{conj}

\begin{cor}
The condition $\left(\frac{a}{p}\right) = -1$ in Conjecture \ref{conj:1}(b) is equivalent to the existence of the 
multiplier $\frac{1+\mu}{2} \in \textrm{End}(E_3(\alpha))$, where $\mu^2 = -3p$ in $\textrm{End}(E_3(\alpha))$ 
and $(\alpha, \beta)$ is the solution of $Fer_3$ corresponding to the factor $u(t) = t^2+at-27a$ of $W_{(p-1)/3}(1-t/27)$.
\label{cor:5}
\end{cor}

We have seen above that part (a) of this conjecture is true for the factors $H_{-8}(X), H_{-11}(X)$, and $H_{-35}(X)$ in \eqref{eqn:4.7}, when these
factors occur, i.e., when the corresponding $\delta$-exponent is $1$.  We can similarly check the factors $H_{-20}(X), H_{-32}(X)$ in part (b).  
Note first that for $p \equiv 1$ (mod $4$), 
$H_{-8}(X)$ still has $\chi_p(a) = -1$; but $H_{-11}(X)$ has $\chi_p(a) = \chi_p(-64) = +1$, corresponding to the fact that it only divides $H_{-12p}(X)$, 
when it does occur.  For the same reasons, $H_{-35}(X)$ has $\chi_p(a) = +1$ for primes $p \equiv 1$ (mod $4$).  On the other hand, there are 
two $a$-values corresponding to the factors $H_{-20}(X), H_{-32}(X)$, as we can see from the computations of \cite[p. 272]{Mor11}:
\begin{align*}
H_{-20}(X) : \ & \ (x^2+(16+88\varepsilon)x-432-2376\varepsilon),\\
& \   (x^2+(16-88\varepsilon)x-432+2376\varepsilon),\\
& \ \varepsilon = \sqrt{-1}, \ a_1 = 16+88\varepsilon, \ a_2 = 16-88\varepsilon,\\
&  \ a_1 a_2 = 16^2+88^2 = 2^6 \cdot 5^3.
\end{align*}
These factors occur when $\left(\frac{5}{p}\right) = -1$, so that $\chi_p(a_1a_2) = -1$.  Hence, $\chi_p(a_1)$ and $\chi_p(a_2)$ have opposite signs,
agreeing with the fact that $H_{-20}(X)^2$ appears in both \eqref{eqn:4.10} and \eqref{eqn:4.11}.  In the same way we have:
\begin{align*}
H_{-32}(X) : \ & \ (x^2+(146-322\varepsilon)x-3942+8694\varepsilon),\\
& \  (x^2+(146+322\varepsilon)x-3942-8694\varepsilon),\\
& \ \varepsilon = \sqrt{-1}, \ a_1 = 146-322\varepsilon, \ a_2 = 146+322\varepsilon,\\
&  \ a_1 a_2 = 146^2+322^2 = 2^3 \cdot 5^6, \ \chi_p(a_1a_2) = -1, \ \textrm{if} \ \left(\frac{-2}{p}\right) = -1.
\end{align*}

\noindent {\bf Example.} We verify the statements of Conjecture \ref{conj:1}(b) for the prime $p = 241$.  The supersingular polynomial (see 
\cite[Thm. 3]{BM04}) is 
\begin{align*}
ss_{241}(x) &\equiv (432)^{20}P_{20}^{(-1/3,-1/2)}\left(1-x/864\right)\\
&\equiv (x + 1)(x + 25)(x + 148)(x + 177)(x + 213)(x + 233)\\
&\times (x^2 + 11x + 156)(x^2 + 166x + 180)(x^2 + 206x + 211)\\
&\times (x^2 + 111x + 25)(x^2 + 28x + 79)(x^2 + 27x + 44)(x^2 + 160x + 117)\\
& \ (\textrm{mod} \ 241).
\end{align*}
Only the quadratics in the second line of this factorization satisfy $Q_3(r,s) \equiv 0$ (mod $241$).  
Considering these quadratics in turn, we have the following values in \eqref{eqn:4.14}:
\begin{align*}
x^2 + 11x + 156 :& \ a \equiv 73, \ \chi_p(a) = -1,\\
x^2 + 166x + 180 :& \ a \equiv 41, \ \chi_p(a) = +1,\\
 x^2 + 206x + 211 :& \ a \equiv 47, \ \chi_p(a) = +1.
 \end{align*}
 Thus, we claim that
 \begin{align*}
 &H_{-3 \cdot 241}(X) \equiv (X^2 + 11X + 156)^2 \ (\textrm{mod} \ 241);\\
&H_{-12 \cdot 241}(X) \equiv (X+32768)^4 (X^2 + 166X + 180)^2  (X^2 + 206X + 211)^2\\
& \ (\textrm{mod} \ 241).
 \end{align*}
Indeed, we do have $h(-3 \cdot 241) = 4, h(-12 \cdot 241) = 12$, so the degrees are correct.  Note that none of the class equations
$H_{-d}(X)$, for $-d \in\{-8, -20, -32, -35\}$ occur in \eqref{eqn:4.10} or \eqref{eqn:4.11} for this prime.  To verify that these congruences are
correct, we only need to check that $X^2+11X+156$ does indeed divide $H_{-3\cdot 241}(X)$ mod $241$ (because $h(-3 \cdot 241) = 4$).  We use
the following criterion from \cite{Mor14}: for this to be true, the map $\mu$ on the $X$-coordinates of points in $E_3(\alpha)[2]$ must be the 
identity map.  This map is given by
$$x^\mu = -\frac{\alpha^2}{36\alpha^{2p}} \left(\frac{\alpha x + 3}{x}\right)^{2p} = -\left(\frac{\beta}{6\alpha}\frac{\alpha x + 3}{x}\right)^{2p} \ \textrm{in} \ \mathbb{F}_{p^2}.$$
We take a root of $u(t) = t^2 +73t +198$ to be $\alpha^3 = 84+6\sqrt{70} \in \mathbb{F}_{p^2}$ and then $\alpha = 145 + 98\sqrt{70}$, 
so that $(\alpha, \alpha^p)$ is a point on $Fer_3$.  Next, the roots of $4x^3+(\alpha x+1)^2$ are
$$x_1 = 187 + 147\sqrt{70}, \ \ x_2 = 65 + 14\sqrt{70}, \ \ x_3 = 2 + 205\sqrt{70}.$$
It is easily checked that $x_i^\mu = x_i$ for all $i$, and this proves that the factorization of $H_{-3\cdot 241}(X)$ modulo $241$ is correct.  Hence, Conjecture
\ref{conj:1} is true for $p = 241$.
\medskip

We will prove Conjecture \ref{conj:1} in Section 5.

\begin{thm} Assume $p \equiv 1$ (mod $3$).
Conjecture \ref{conj:1} implies that the polynomial $g_p(x)$ has $h(-3p)/2$ linear factors over $\mathbb{F}_p$ if $p \equiv 1$ (mod $4$),
 and no linear factors if $p \equiv 3$ (mod $4$).
 \label{thm:13}
 \end{thm}
 
 \begin{proof}
 By the remark at the beginning of this subsection, $\chi_p(a) = \chi_p\left(\frac{-108}{b+1}\right) = \chi_p(b+1)$, for every 
 binomial quadratic $x^2+b$ dividing $P_{(p-1)/3}(x)$ (mod $p$).  If $p \equiv 3$ (mod $4$), then Conjecture \ref{conj:1}(a) implies
that $\chi_p(b+1) = \chi_p(a) = -1$, for every such binomial quadratic.  But then $\chi_p(b) \neq \chi_p(b+1)$
 in Proposition \ref{prop:9}(b), so there are no roots of $g_p(x)$ in $\mathbb{F}_p$, when $p \equiv 7$ (mod $12$). \medskip
 
 On the other hand, if $p \equiv 1$ (mod $4$), then Conjecture \ref{conj:1}(b) implies that 
 $$\chi_p(b+1) = \chi_p(a) = -1 = \chi_p(b)$$
 if and only if $a$ corresponds to a quadratic factor $x^2+ax-27a$ of $W_{(p-1)/3}(1-x/27)$ coming from a factor $q(X) = X^2+rX+s$
 of $H_{-3p}(X)$; by the above discussion, this is also true for the quadratic $x^2-8x+216$ corresponding to the factor $X-8000$, 
 when $\left(\frac{-2}{p}\right) = -1$.  This gives one quadratic factor for each power of an irreducible factor in \eqref{eqn:4.10}.  Each such 
 power has degree $4$, so there are a total of $h(-3p)/4$ binomial quadratics $x^2+b$ dividing $P_{(p-1)/3}(x)$ for which 
 $\chi_p(b+1) = \chi_p(b) = -1$.  Now Proposition \ref{prop:9}(b) implies that the number of linear factors of $g_p(x)$ is twice that number,
 i.e., $h(-3p)/2$, when $p \equiv 1$ (mod $12$).  This proves the theorem for primes $p > 53$, and it can be checked directly for smaller primes. There are
 only two primes $p < 53$ satisfying $p \equiv 1$ (mod $12$); namely, $p = 13, 37$.  For these primes we have:
 \begin{align*}
 p = 13: \ g_{13}(x) &\equiv 2(x + 5)(x + 8)(x^2 + 8x + 1), \ h(-39) = 4;\\
 p = 37: \ g_{37}(x) &\equiv 27(x + 6)(x + 23)(x + 29)(x + 31)\\
 &\times (x^2 + 7x + 1)(x^2 + 19x + 1)(x^2 + x + 26)(x^2 + 10x + 10),\\
 h(-111) &= 8.
 \end{align*}
 For the other primes $p<53$ with $p \equiv 7$ (mod $12$), namely, $p = 7, 19, 31, 43$, $g_p(x)$ has no linear factors modulo $p$.
 \end{proof}
 
 \subsection{Primes $p \equiv 2$ modulo $3$.}
 
 \begin{thm}
 Assume $p \equiv 2$ mod $3$ and that Conjecture \ref{conj:1} holds for $p$.  Then the number of distinct linear factors of 
 $g_p(x)$ (mod $p$) is equal to
\begin{equation*}
N_1(g,p) = \begin{cases} h(-p) + \frac{3h(-3p)}{2} - 2, & p \equiv 1 \ \textrm{mod} \ 8,\\
3h(-p) + \frac{h(-3p)}{2} - 2, & p \equiv 3 \ \textrm{mod} \ 8,\\
h(-p) + \frac{h(-3p)}{2} - 2, & p \equiv 5 \ \textrm{or} \ 7 \ \textrm{mod} \ 8.
\end{cases}
\end{equation*}
\label{thm:14}
\end{thm}

\begin{proof}
When $p \equiv 2$ mod $3$, the linear factors of $g_p(x)$ over $\mathbb{F}_p$ come from two sources. \medskip

(i) If $\rho \in \mathbb{F}_p$ is a root of $g_p(x)$ for which $\rho = v^2$ in $\mathbb{F}_p$, then by \eqref{eqn:4.4} the polynomial 
$W_{(p-2)/3}(x)$ has the root $\left(\frac{v-1}{v+1}\right)^2$, which is also a square in $\mathbb{F}_p$.  
Conversely, if $\gamma^2 = \left(\frac{v-1}{v+1}\right)^2$
is a root of $W_{(p-2)/3}(x)$, where $\gamma \in \mathbb{F}_p$, then by \eqref{eqn:4.4}, $v^2 \in (\mathbb{F}_p^\times)^2$ is a root of
$g_p(x)$.  As in the proof of Theorem \ref{thm:8}, this sets up a 1--1 correspondence between the roots of $g_p(x)$ which are 
squares and the roots of $W_{(p-2)/3}(x)$ which are squares.  Hence, by \cite[Thm. 1(d)]{BM04} and Proposition \ref{prop:6}, 
the number of such roots (including $-1$ when $p \equiv 1$ mod $4$) is 
\begin{equation}
N_1^{(+)}(g,p) = \begin{cases} h(-p)-1, & \ p \equiv 1 \ (\textrm{mod} \ 4),\\
3h(-p)-1, & \ p \equiv 3 \ (\textrm{mod} \ 8),\\
h(-p)-1, & \ p \equiv 7 \ (\textrm{mod} \ 8).
\end{cases}
\label{eqn:4.15}
\end{equation}

(ii) As in the proof of Proposition \ref{prop:9}, the remaining roots $\alpha \neq -1$ of $g_p(x)$ give rise to irreducible factors $u^2+c$ of
$g_p(u^2)$, and these arise from binomial quadratic factors $x^2+b$ of the Legendre polynomial $P_{(p-2)/3}(x)$, by means of \eqref{eqn:4.2}.
 Then $b,c$ satisfy the relation
 $$b = \frac{(c-1)^2}{4c}, \ b+1 = \frac{(c+1)^2}{4c}, \ \chi_p(-c) = -1,$$
 where 
\begin{align*}
\chi_p(c) = \chi_p(b) &= \chi_p(b+1) = -1, \ p \equiv 1 \ (\textrm{mod} \ 4),\\
\chi_p(c) = \chi_p(b) &= \chi_p(b+1) = +1, \ p \equiv 3 \ (\textrm{mod} \ 4).
\end{align*}
Furthermore, as in Proposition \ref{prop:9}(c), the factors $x^2+b$ of $P_{(p-2)/3}(x)$ are in 1--1 correspondence with the irreducible
quadratic factors of $W_{(p-2)/3}(1-t/27)$ of the form
$$t^2 + at -27a, \ \textrm{with} \ a = \frac{-108}{b+1}.$$
But here we have $\left(\frac{-3}{p}\right) = \left(\frac{p}{3}\right) = -1$, so that the condition on $a$ is
\begin{equation}
\left(\frac{a}{p}\right) = -\chi_p(b+1) = \begin{cases} +1, & \  p \equiv 1 \ (\textrm{mod} \ 4),\\
-1, &  p \equiv 3 \ (\textrm{mod} \ 4). \end{cases}
\label{eqn:4.16}
\end{equation}
When $p \equiv 3$ mod $4$, i.e., $p \equiv 11$ (mod $12$), the class equation $H_{-12p}(X)$ satisfies the congruence
 \begin{align}
\notag H_{-12p}(X) &\equiv X^2 H_{-12}(X)^2 H_{-8}(X)^{4\delta_2} H_{-11}(X)^{4\delta_3}H_{-35}(X)^{4\delta_6}\\
 & \ \times \prod_{i}{(X^2+r_iX+s_i)^2} \ \textrm{mod} \ p,
 \label{eqn:4.17}
 \end{align}
only differing from \eqref{eqn:4.7} in the first two factors.  Conjecture \ref{conj:1} says that all the irreducible quadratic factors
$$q(X) = X^2+rX+s = X^2+(a^3+126a^2+2944a)X+a(a-192)^3$$
of $H_{-12p}(X)$ (whose roots are $j$-invariants for the maximal order of $\mathbb{Q}(\sqrt{-3p})$ in this case) 
satisfy $\left(\frac{a}{p}\right) = -1$.  Note that $X^2$ and $H_{-12}(X)^2 = (X-54000)^2$ do not contribute factors of the form
$t^2+at-27a$, by the discussion in \cite[p. 270]{Mor11}.  Also, each $4$-th power of an irreducible (linear or quadratic) 
contributes $1$ or $2$ factors of the form $t^2+at-27a$ to the factorization of $W_{(p-2)/3}(1-t/27)$, 
while the square factors $q(X)$ in the product contribute one factor each.
  Hence, there are a total of $\frac{h(-3p)-4}{4}$ binomial quadratic factors $x^2+b$ satisfying \eqref{eqn:4.16}.  Since the correspondence between
  $b$ and $c$ is 1--2, this gives that the total number of linear factors of $g_p(x)$ (other than $x+1$) arising from binomial quadratics is equal to
  \begin{equation}
  N_1^{(-)}(g,p) = \frac{h(-12p)-4}{2} = \frac{h(-3p)}{2}-2, \ p \equiv 3 \ (\textrm{mod} \ 4).
  \label{eqn:4.18}
  \end{equation}
  
  On the other hand, if $p \equiv 1$ mod $4$, then as in \eqref{eqn:4.10} and \eqref{eqn:4.11}, there are two class equations to consider:
  $H_{-3p}(X)$ corresponds to the maximal order of $\mathbb{Q}(\sqrt{-3p})$, while $H_{-12p}(X)$ corresponds to the ring
  $\textsf{R}_{-12p} \subset \mathbb{Q}(\sqrt{-3p})$ of discriminant $-12p$ and conductor $f = 2$.  
  As above, the factor $X^2$ needs to be added to the right side of \eqref{eqn:4.10}, and $(X-54000)^2$ to the right side of \eqref{eqn:4.11}, 
  but these do not contribute to the count we are interested in.  Now, the condition \eqref{eqn:4.16} requires that
  $\chi_p(a) = +1$, and Conjecture \ref{conj:1} says this is satisfied exactly by the factors of $H_{-12p}(X)$.  Now the number of linear
  factors of $g_p(x)$ coming from the binomial quadratics $x^2+b$ is, with $\textsf{h}(-12p) = \textrm{deg}(H_{-12p}(X))$, equal to
  \begin{equation}
  N_1^{(-)}(g,p) = \frac{\textsf{h}(-12p)-2}{2} = \begin{cases} \frac{h(-3p)}{2}-1, \ p \equiv 5 \ (\textrm{mod} \ 8),\\
   \frac{3h(-3p)}{2}-1, \ p \equiv 1 \ (\textrm{mod} \ 8). \end{cases}
  \label{eqn:4.19}
  \end{equation}
 This holds because the Kronecker symbol $\left(\frac{-3p}{2}\right) = +1$ if $p \equiv 5$ mod $8$ and is $-1$ otherwise. (See \cite[p.132]{Co13}.)
  Altogether, then, and adding $1$ for the root $-1$ when $p \equiv 3$ (mod $4$) 
  ($-1$ is not a square mod $p$ and $x^2+1$ does not arise from a binomial quadratic factor), 
  the total number of linear factors of $g_p(x)$ is equal to
 \begin{align*}
 N_1(g,p) &= N_1^{(+)}(g,p) + N_1^{(-)}(g,p) + \frac{1}{2}\left(1-(-1)^{(p-1)/2} \right)\\
   &= \begin{cases} h(-p) + \frac{3h(-3p)}{2} - 2, & p \equiv 1 \ (\textrm{mod} \ 8),\\
   3h(-p) + \frac{h(-3p)}{2}- 2, & p \equiv 3 \ (\textrm{mod} \ 8),\\
   h(-p) + \frac{h(-3p)}{2} - 2, & p \equiv 5, 7 \ (\textrm{mod} \ 8). \end{cases}
 \end{align*}
This proves the theorem for $p > 53$.  For $p \le 53$, the number of linear factors of $g_p(x)$ mod $p$ is given in Table \ref{tab:1}
along with the class numbers $h(-p),h(-3p)$.  These values satisfy the formulas given in the theorem.
\end{proof}

\begin{table}
  \centering 
  \caption{$N_1(g,p)$ for primes $p \equiv 2$ mod $3$.}\label{tab:1}
  
\noindent \begin{tabular}{|c|ccc|}
\hline
   &    &  & \\
$p$	&   $N_1(g,p)$ & $h(-p)$ & $h(-3p)$ \\
   &    &  & \\
\hline
5 & 1 & 2 & 2\\
11  & 3 & 1 & 4 \\
17 & 5 & 4& 2\\ 
23 & 5 & 3 & 8\\
 29 & 7 & 6 & 6\\
 41 & 9 & 8 & 2\\
  47 & 7 & 5 & 8\\
  53 & 9 & 6 & 10\\
  \hline
\end{tabular}
\end{table}

\begin{cor} If $p \equiv 2$ (mod $3$) is a prime, then $h(-p)+\frac{h(-3p)}{2}$ is always odd.
\label{cor:6}
\end{cor}

It is clear from the proof of Theorem \ref{thm:14} that the terms involving $h(-p)$ account for roots of $g_p(x)$ which are
squares in $\mathbb{F}_p$, and the terms involving $h(-3p)$ come from roots which are nonsquares in $\mathbb{F}_p$.  This is
the same situation that was stated in Theorem \ref{thm:11} for the roots of $h_p(x)$, for $p \equiv 3$ mod $4$.
\medskip

Corresponding to Theorem \ref{thm:B} we have the following result for the polynomial $g_p(x)$.  Again, we assume the truth
of Theorem \ref{thm:D} (Conjecture \ref{conj:1}), which we prove in the next section.

\begin{thm}
\noindent (a) Assume $p \equiv 1$ (mod $3$).  The number of irreducible quadratics of the form $x^2+ax+1$ which divide $g_p(x)$ mod $p$ is
\begin{equation*}
N_2(g,p) = \begin{cases} \frac{3h(-3p)}{4}, & \ p \equiv 1 \ (\textrm{mod} \ 8),\\
\frac{h(-3p)}{4}, & \ p \equiv 5 \ (\textrm{mod} \ 8),\\
\frac{h(-3p)}{4},  & \ p \equiv 3 \ (\textrm{mod} \ 4).
\end{cases}
\end{equation*}

\noindent (b) Assume $p \equiv 2$ (mod $3$). The number of irreducible quadratics of the form $x^2+ax+1$ which divide $g_p(x)$ mod $p$ is
\begin{equation*}
N_2(g,p) = \begin{cases} \frac{h(-3p)-2}{4}, & \  p \equiv 1 \ (\textrm{mod} \ 4),\\
\frac{h(-p)-1}{2},  & \  p \equiv 3 \ (\textrm{mod} \ 4). \end{cases}
\end{equation*}
\label{thm:15a}
\end{thm}

\begin{proof}
(a) Assume $p \equiv 1$ (mod $3$) and $p > 53$. This case corresponds to the situation in Theorem \ref{thm:9}.  Here there are no linear factors of $P_{(p-1)/3}(x)$ 
by \cite[Thm. 1(d)]{BM04}, so we only have to show that a binomial quadratic factor $x^2+b$ of 
$P_{(p-1)/3}(x)$ corresponds, by \eqref{eqn:4.2}, to an irreducible factor 
$x^2+(4b+2)x+1$ of $g_p(x)$.  The discriminant of the latter polynomial is $16b(b+1)$, and we know that $\chi_p(-b) = -1$, hence
$$\chi_p(b) = \begin{cases} -1, & \ p \equiv 1 \ (\textrm{mod} \ 4),\\
+1, & \ p \equiv 3 \ (\textrm{mod} \ 4). \end{cases} $$
This is where Theorem \ref{thm:D} comes into play.  Each binomial quadratic $x^2+b$ corresponds, by the argument in \cite[\S 5]{Mor11}, 
to a unique irreducible factor of $W_{(p-1)/3}(1-x/27)$ of the form
$$x^2-\frac{108}{b+1}x +\frac{2916}{b+1} = x^2+ax-27a.$$
If $p \equiv 3$ mod $4$, then Theorem \ref{thm:D}(a) and the discussion in Section 4.2 give that $\chi_p(b+1) = \chi_p(a) = -1$.  
Hence, $\chi_p(16b(b+1)) = -1$, giving that $x^2+(4b+2)x+1$ is indeed irreducible.  Thus,
$N_2(g,p)$ is just the number of binomial quadratic factors of $P_{(p-1)/3}(x)$, which is $\frac{h(-3p)}{4}$, by \cite[Thm. 1.1]{Mor11}. \medskip

On the other hand, if $p \equiv 1$ (mod $4$), $\chi_p(b+1) = \chi_p(a) = +1$ if and only if the factor $q(X)$ in Theorem \ref{thm:D}(b) divides $H_{-12p}(X)$
mod $p$, corresponding to the order $\textsf{R}_{-12p}$ of conductor $2$.  If $\textsf{h}(-12p) = \textrm{deg}(H_{-12p}(X))$
 denotes the class number of this order, then there are $\textsf{h}(-12p)/4$ binomial quadratic factors for which $x^2+(4b+2)x+1$ is irreducible
 (by the argument in the proof of Theorem \ref{thm:13}, applied to \eqref{eqn:4.11});
  and this number equals $3h(-3p)/4$ or $h(-3p)/4$, according as $p \equiv 1$ or $5$ (mod $8$).  This proves (a). \medskip
 
 \noindent (b) Now assume $p \equiv 2$ (mod $3$) and $p > 53$.  As in the proof of Theorem \ref{thm:12}, factors of the form $x^2+bx+1$ arise from two possible
 sources in the congruence \eqref{eqn:4.4}: from linear factors $x-a$ of $W_{(p-2)/3}(x)$ for which $\chi_p(a) = -1$, and quadratic factors $x^2+ax+1$ of  
 $W_{(p-2)/3}(x)$ for which \eqref{eqn:3.14} has a solution $c \in \mathbb{F}_p$, i.e., for which 
 $$\chi_p(a+2) = \left(\frac{-1}{p}\right).$$
 Proposition \ref{prop:6} shows that no linear factor contributes a quadratic
 of this form, if $p \equiv 5$ (mod $12$); while pairs of reciprocal roots $\{a,1/a\}$ contribute $\frac{h(-p)-1}{2}$ such quadratics,
  if $p \equiv 11$ (mod $12$). \medskip
 
 Setting $a = 2\frac{1-k}{1+k}$, an irreducible factor $x^2+2\frac{1-k}{1+k}x+1$ of $W_{(p-2)/3}(x)$ corresponds to the factor
 $x^2-\frac{108}{k+1}x+\frac{2916}{k+1}$ of $W_{(p-2)/3}(1-x/27)$ (see \cite[Eq. (5.3)]{Mor11}).  Then we have that
 $$a+2 = 2\frac{1-k}{1+k} + 2 = \frac{4}{k+1}.$$
 Hence, 
 $$\chi_p(a+2) = \chi_p(k+1) = -\chi_p(\tilde a), \ \tilde a = \frac{-108}{k+1},$$
 and by Theorem \ref{thm:D}, 
 $$-\chi_p(\tilde a) = \begin{cases} +1, & \ p \equiv 3 \ (\textrm{mod} \ 4),\\
 +1, & \ p \equiv 1 \ (\textrm{mod} \ 4), \ q(X) = X^2+rX+s \mid H_{-3p}(X),\\
 -1, & \ p \equiv 1 \ (\textrm{mod} \ 4), \ q(X) = X^2+rX+s \mid H_{-12p}(X), \end{cases}$$
 where $q(X)$ is given by \eqref{eqn:4.14} with $\tilde a$ for $a$.  It follows that no factors $x^2+ax+1$ of $W_{(p-2)/3}(x)$
 contribute to $N_2(g,p)$ when $p \equiv 11$  (mod $12$), and only factors corresponding to factors $q(X)$ of $H_{-3p}(X)$ 
 contribute to $N_2(g,p)$, when $p \equiv 5$ (mod $12$).  Now the formula for $N_2(g,p)$ follows from Proposition \ref{prop:6}(c), 
 as in the last paragraph in the proof of Theorem \ref{thm:12} for $p \equiv 3$ (mod $4$); and from \eqref{eqn:1.4}, if $p \equiv 1$ (mod $4$). \medskip
 
 The counts in the theorem can be checked by direct calculation for primes $p \le 53$. 
\end{proof}

\begin{cor}
Except for the primes in the set
$$A = \{5, 11, 17, 41, 89\}$$
$g_p(x)$ always has an irreducible factor of the form $x^2+ax+1$ over $\mathbb{F}_p$.  In particular, $g_p(x)$ splits completely into
linear factors (mod $p$) if and only if $p \in \{5, 11, 17\}$.
\label{cor:7}
\end{cor}

\begin{proof}
If $p \equiv 1$ (mod $3$), it is clear from part (a) of the theorem that $N_2(g,p) \ge 1$.  If $p \equiv 2$ (mod $3$), the only primes
$p \ge 5$ for which $N_2(g,p) = 0$ are the primes congruent to $5$ (mod $12$), for which $h(-3p) = 2$, and the primes congruent to 
$11$ (mod $12$), for which $h(-p) = 1$.  These are exactly the primes in $A$, as can be verified by consulting a table of 
imaginary class numbers for $3p < 427$, since $-d = -427$ is the last discriminant with $h(-d) = 2$. See \cite{Wat04} and the references therein.
Then $g_p(x)$ splits into linear factors modulo $p$ only for the first three primes in $A$.
\end{proof}

Theorem \ref{thm:B} implies a similar result for primes $p$ for which $h_p(x)$ splits into linear factors (mod $p$).  However, the case $p \equiv 3$ (mod $8$)
requires extra work, since $h_p(x)$ has no factors of the form $x^2+ax+1$ in this case.

\begin{prop}
The only primes $p \ge 5$, for which $h_p(x)$ splits into linear factors (mod $p$), are the primes $p \in \{5, 7, 11,19\}$.
\label{prop:10}
\end{prop}

\begin{proof}
Theorem \ref{thm:B} shows that $h_p(x)$ can split into linear factors only when $p \equiv 5$ (mod $8$) and $h(-2p) = 2$; $p \equiv 7$ (mod $8$) and 
$h(-p) = 1$; or $p \equiv 3$ (mod $8$).  The first two cases yield only $p = 5, 29$ and $7$, respectively; but $h_{29}(x) $ has two quadratic factors (mod $29$).  
For the third case, we use Corollary \ref{cor:A1} when $p \equiv 3$ (mod $8$),
noting that $h_p(x)$ splits completely if and only if $\textrm{deg}(h_p(x)) = N_1(h,p)$, which is equivalent to
\begin{equation}
p = 12h(-p) + 2h(-2p) -5.
\label{eqn:4.20}
\end{equation}
This does hold for $p = 11$ and $19$.  To show that it does not hold for larger primes, we make use of the polynomials
$u_d(x)$ listed in Table \ref{tab:2}, whose irreducible factors divide $h_p(x)$ 
whenever $p$ satisfies $\left(\frac{d}{p}\right) = 0$ or $-1$. \medskip

\begin{table}
  \centering 
  \caption{Specific factors of $h_p(x)$ for $(d/p) = -1$.}\label{tab:2}
  
\noindent \begin{tabular}{|c|c|c|}
\hline
   &    &  \\
$d$	 & $u_d(x)$ & $\textrm{disc}(u_d(x))$ \\
   &   &  \\
\hline
-20 &  $x^2-3x+1$ & $5$\\
-24  &   $9x^2-14x+9$ & $-2^7$\\
-32 &   $2401x^4-126596x^3+232006x^2-126596x+2401$ & $-2^{52}5^{12}7^2 13^4$\\ 
-35 &   $923521x^6 -815736374x^5 +3045233215x^4 $ & $2^{136}5^{15}7^{10}19^4$\\
&   $-4309928180x^3 +2965797615x^2 $ & $\cdot 23^4 31^2 43^4$\\
& $-1198928374x +312900721$ & $\cdot 59^2 67^2$\\
 -84 &  $81x^4-2070x^3+3979x^2-2070x+81$ & $2^{16}3^{10}7^6 13^2$\\
 -175 & $319185859129860321x^6 - 326391621585356934x^5$ & $2^{150}3^{64} 5^5 7^{12}$\\
 & $-60682931604620945x^4 - 705896489840058900x^3 $ & $\cdot 13^{12}17^6 a^2$ \\
 & $1434431523926324655x^2 - 756097015482369414x $ & $p \mid a$ \\
 & $+95450676529963041$ & $\Rightarrow p \le 307$\\
  \hline
\end{tabular}
\end{table}

The existence of these factors can be seen as follows.  We start with the fact that the roots of $W_{(p-1)/2}(x)$ are 
supersingular parameters for the Legendre normal form
$$E_2: \ Y^2 = X(X-1)(X-\lambda), \ \ j =  j(E_2) = \frac{2^8(\lambda^2-\lambda+1)^3}{\lambda^2(\lambda-1)^2}$$
in characteristic $p$.  Then setting $t = 4\lambda(1-\lambda)$, the $j$-invariant of the elliptic curve $E_2$ is
$$j = -\frac{64(t-4)^3}{t^2},$$
in terms of the roots $t$ of $W_{(p-e)/4}(x)$, by \eqref{eqn:2.5} and \eqref{eqn:3.8}.  Now put $t = \frac{(u-1)^2}{(u+1)^2}$.  This gives that
$$j = \frac{64(u+3)^3(3u+1)^3}{(u-1)^4(u+1)^2},$$
where $u$ is a root of $h_p(u^2)$, by \eqref{eqn:3.5}.  Then the polynomials in Table \ref{tab:2} occur as (or are derivable from) factors of the resultant
\begin{equation}
R_{d}(u) = \textrm{Res}_j\left(H_{d}(j),(u-1)^4(u+1)^2 j - 64(u+3)^3(3u+1)^3\right).
\label{eqn:}
\end{equation}
For example, with $H_{-20}(j) = j^2-1264000j-681472000$, we find\footnote{See Section 4.2 or \cite{Mor14} for the class equations 
$H_d(X)$, for $d \in \{-20, -32, -35\}$; and \cite{Mor21} for $d \in \{-24,-84\}$.  $H_{-175}(X)$ can be computed using \cite[Sec. 12]{LM15}.  Also see \cite[Table 2]{Mor19}.} that
\begin{align*}
R_{-20}(u) &= -2^{16}(u^2+u-1)(u^2-u-1) (43681u^8 + 224120u^7 + 655004u^6 \\
& \ + 37960u^5 - 872954u^4 + 37960u^3 + 655004u^2 + 224120u + 43681).
\end{align*}
This shows that whenever $\left(\frac{-20}{p}\right) = -1$, the factor $(u^2+u-1)(u^2-u-1) = u^4-3u^2+1$ divides $h_p(u^2)$, and therefore
$x^2-3x+1$ divides $h_p(x)$.  This factor is only irreducible when $\left(\frac{5}{p}\right) = -1$, which requires that $p \equiv 1$ (mod $4$).
If $p \equiv 3$ (mod $4$), then $x^2-3x+1$ yields two linear factors of $h_p(x)$ mod $p$.  
On the other hand, the $8$-th degree factor $k(u)$ of $R_{-20}(u)$ yields the factor of $h_p(x)$ given by
\begin{align*}
L(x) &= k(\sqrt{x})k(-\sqrt{x}) = 1908029761x^8 + 6992685048x^7 + 335752042268x^6\\
& \ - 1104810416184x^5 + 1520583753670x^4 - 1104810416184x^3 \\
& \ + 335752042268x^2 + 6992685048x + 1908029761,
\end{align*}
an irreducible polynomial (over $\mathbb{Q}$) with square discriminant
$$D = 2^{244}5^{32}11^{12}13^{16}17^8 19^4 31^6 37^4.$$
The PSV theorem implies that $L(x)$ always factors into a product of an even number of factors (mod $p$) (when $p \nmid D$), so that
$L(x)$ has irreducible quadratic factors whenever $\left(\frac{-5}{p}\right) = -1$ and it does not split completely.  In fact, 
$L(x)$ does not split completely for any prime $p \equiv 3$ (mod $8$) for which $p >19$, because of the following fact. 
As a reciprocal polynomial, $L(x) = x^4 M\left(x+\frac{1}{x}\right)$, where
\begin{align*}
M(y) &= 1908029761y^4 + 6992685048y^3 + 328119923224y^2\\
& \ \  - 1125788471328y + 852895728656,\\
\textrm{disc}(M(y)) &= -2^{104}5^{14}11^6 13^6 17^4 19^2 31^2 37^2.
\end{align*}
The PSV theorem implies that $M(y)$ has an odd number of irreducible factors (mod $p$), when $p \equiv 3$ (mod $8$) and $p \nmid D$.  It cannot be
irreducible, since $h_p(x)$ has no irreducible quartic factors, so it must be divisible by an irreducible quadratic.  This yields two irreducible quadratic
factors of $L(x)$ whenever $p > 19$ and $\left(\frac{5}{p}\right) = +1$. 
\medskip

Next, the sixth degree factor $u_{-35}(x)$ for $d = -35$ also provides at least one irreducible quadratic factor when 
$\left(\frac{5}{p}\right) = -1$ and $\left(\frac{-7}{p}\right) = +1$.  The PSV theorem implies then that $u_{-35}(x)$ has an odd number
of irreducible factors (mod $p$), when $p \nmid \textrm{disc}(u_{-35}(x))$, so that $u_{-35}(x)$ cannot split completely into linear factors.  
By the discriminant listed in Table \ref{tab:2}, this holds for primes $p \neq 43, 67$, and for 
these primes it can be checked directly that $u_{-35}(x)$ has irreducible quadratic factors mod $p$:
\begin{equation*}
u_{-35}(x)  \equiv \begin{cases} 10(x^2 + 5x + 10)(x + 13)^2(x + 16)^2 \ (\textrm{mod} \ 43);\\
60(x^2 + 47x + 54)(x^2 + 53x + 36)(x + 35)^2  \ (\textrm{mod} \ 67).
 \end{cases}
\end{equation*}

It remains to consider primes with $p \equiv 3$ (mod $8$) and $\left(\frac{p}{5}\right) = \left(\frac{p}{7}\right) = -1$.  Most of these primes can be eliminated by
considering the factor $u_{-175}(x)$ in Table \ref{tab:2} corresponding to the discriminant $d = -7 \cdot 5^2$.  Note that the integer $a^2$ listed in
Table \ref{tab:2} as a factor of the discriminant $\tilde D =  \textrm{disc}(u_{-175}(x))$ is given by
$$a^2 = 19^4 31^4 41^4 47^4 59^2 83^2 131^2 199^2 223^2 251^2 271^2 283^2 307^2.$$
The condition that the irreducible factors of $u_{-175}(x)$ divide
$h_p(x)$ is that $\left(\frac{d}{p}\right) = \left(\frac{-7}{p}\right) = \left(\frac{p}{7}\right) = -1$.  If, in addition, $\left(\frac{p}{5}\right) = \left(\frac{5}{p}\right) = -1$,
then because $5^5 \ ||\ \tilde D$, the polynomial $u_{-175}(x)$ must have an odd number of irreducible factors mod $p$, for 
$p \nmid \tilde D$.  The only primes dividing $\tilde D$ satisfying
$p \equiv 3$ (mod $8$) and $\left(\frac{p}{5}\right) = \left(\frac{p}{7}\right) = -1$ are now just $p \in \{83, 283, 307\}$.  It can be checked that
$u_{-175}(x)$ has two irreducible quadratic factors modulo each of these primes.  This completes the proof.
\end{proof}

\section{Proof of Conjecture \ref{conj:1} and Theorem \ref{thm:D}.}

To prove Conjecture \ref{conj:1}, we first find rational expressions for the roots of $4x^3+(\alpha x+1)^2 = 0$.  We will then be able to apply
the criterion for the existence of a multiplier $\frac{1+\mu}{2} \in \textrm{End}(E_3(\alpha))$ that we made use of in the example just before
Theorem \ref{thm:13}.  See \eqref{eqn:5.10} below. \medskip

We use that $E_3(\alpha) \cong E_6(b)$ over $\overline{\mathbb{F}}_p$, for a suitable substitution for $\alpha$ in terms of $b$,
 where $E_6(b)$ is the Tate normal form, on which $(0,0)$ is a point of order $6$:
\begin{align}
\notag E_6(b):  & \ Y^2+bXY-(b-1)(b-2)Y = X^3-(b-1)(b-2)X^2,\\
 \label{eqn:5.1} & \ \alpha^3 = \frac{(3b-4)^3}{(b-2)(b-1)^2}, \ \ b \neq 1,2, \frac{10}{9}.
 \end{align} 
This isomorphism follows from the fact that the $j$-invariants of these two curves are the same, with the indicated substitution for $\alpha^3$:
$$j(E_3(\alpha)) = \frac{\alpha^3(\alpha^3-24)^3}{\alpha^3-27} =  \frac{(3b - 4)^3(3b^3 - 12b^2 + 24b - 16)^3}{(9b - 10)(b - 1)^6(b - 2)^3}= j(E_6(b)).$$
On this curve the point $P = (1-b,(b-1)^2)$ is a point of order $2$, since $-P = P$ follows from the relation $y = -bx+(b-1)(b-2)-y$ between the $y$-coordinates
of $P$ and $-P$.  If
$$G(x,y) = Y^2+bXY-(b-1)(b-2)Y - X^3 + (b-1)(b-2)X^2,$$
then
\begin{align*}
&G\left(x, -\frac{1}{2} bx + \frac{1}{2}b^2-\frac{3}{2}b+1\right) = -\frac{1}{4}(x-1+b)\\
& \ \ \ \ \times (4x^2 + (-3b^2 + 8b - 4)x + b^3 - 5b^2 + 8b - 4).
\end{align*}
The discriminant of this cofactor is $d = (9b-10)(b-2)^3$.  We now parametrize the conic $(9b-10)(b-2) = c^2$ by
$$b = \frac{-t^2+20}{6t+28}, \ \ c = \frac{3t^2+28t+60}{6t+28}, \ \ t \notin \{-2, -4, -6, -\frac{10}{3}, -\frac{14}{3}\}.$$
With this expression for $b$, we find the $x$-coordinates of the points of order $2$ on $E_6(b)$ to be
\begin{align}
\label{eqn:5.2} &\xi_1 = \frac{(t + 4)(t + 2)}{2(3t + 14)}, \ \ \xi_2 = -\frac{(t + 4)(t + 6)^2}{2(3t + 14)^2},\\
\label{eqn:5.3} &\xi_3 = \frac{(t + 2)(t + 6)^2}{16(3t + 14)}.
\end{align}

Then using \eqref{eqn:5.1} we find that
\begin{equation}
\alpha^3 = \frac{(3t^2+24t+52)^3}{(t+2)^2(t+4)^2(t+6)^2}.
\label{eqn:5.4}
\end{equation}
We will show below that $t \in \mathbb{F}_{p^2}$, in the situation we are interested in, but for now we leave this question aside.
Next, we need an explicit isomorphism between $E_3(\alpha)$ and $E_6(b)$.  We find without difficulty that the relation between the 
$X$-coordinates of points of $E_6(b)$ and the $x$-coordinates for the points on $E_3(\alpha)$, given by
\begin{equation*}
X+\frac{1}{12}(-3b^2+12b-8) = \frac{(3b-4)^2}{\alpha^2}\left(x+\frac{\alpha^2}{12}\right),
\end{equation*}
gives rise to an isomorphism.  From this we obtain that
\begin{equation}
x = F(X) = \frac{(X -b^2 + 3b - 2)\alpha^2}{(3b - 4)^2}.
\label{eqn:5.5}
\end{equation}
This allows us to find the $x$-coordinates of the points of order $2$ on $E_3(\alpha)$, namely:
\begin{align}
\label{eqn:5.6} x_1 &= F(\xi_1) = \frac{-\alpha^2(t + 2)^2(t + 4)^2}{(3t^2 + 24t + 52)^2},\\
\label{eqn:5.7} x_2 &= F(\xi_2) = \frac{-\alpha^2(t + 4)^2(t + 6)^2}{(3t^2 + 24t + 52)^2},\\
\label{eqn:5.8} x_3 &= F(\xi_3) = \frac{-\alpha^2(t + 2)^2(t + 6)^2}{4(3t^2 + 24t + 52)^2}.
\end{align}
As a check, we compute the elementary symmetric functions of the $x_i$, giving:
\begin{align*}
x_1 + x_2 + x_3 &= \frac{-\alpha^2}{4},\\
x_1 x_2 + x_1 x_3 + x_2 x_3 &= \frac{\alpha^4(t + 2)^2(t + 4)^2(t + 6)^2}{2(3t^2 + 24t + 52)^3} = \frac{\alpha}{2},\\
x_1 x_2 x_3 &=  -\frac{\alpha^6(t + 2)^4(t + 4)^4(t + 6)^4}{4(3t^2 + 24t + 52)^6} = \frac{-1}{4};
\end{align*}
where we have used \eqref{eqn:5.4} and the relation
$$4(t+2)^2(t+4)^2+4(t+4)^2(t+6)^2+(t+2)^2(t+6)^2 = (3t^2+24t+52)^2.$$
These relations verify that the $x_i$ are the roots of $4x^3+(\alpha x+1)^2 = 0$. \medskip

Now from Lemma \ref{lem:4} we know that the $x_i$ lie in $\mathbb{F}_{p^2}$.  From \cite{Mor14} we have a multiplier (meromorphism) $\mu$ on the function
field $\textsf{K} = \overline{\mathbb{F}}_p(x,y)$ of the curve $E_3(\alpha)$ (equivalently, an endomorphism of $E_3(\alpha)$) 
for which $\mu^2 = -3p$, and which acts as a permutation on $x$-coordinates of nontrivial points of order $2$ by
\begin{equation}
x^\mu = -\frac{\alpha^2}{36\alpha^{2p}} \left(\frac{\alpha x + 3}{x}\right)^{2p} = -\left(\frac{\beta}{6\alpha}\frac{\alpha x + 3}{x}\right)^{2p} \ \textrm{in} \ \mathbb{F}_{p^2}.
\label{eqn:5.9}
\end{equation}
The criterion proved in \cite[pp. 93-94]{Mor14} is: given a multiplier $\mu = \sqrt{-3p}$ in the multiplier ring of $\textsf{K}$,
\begin{align}
\notag &\textrm{There exists a multiplier of the form} \ \frac{1+\mu}{2} \ \textrm{on} \ \textsf{K}\\
\label{eqn:5.10} & \ \textrm{if and only if} \ \mu \ \textrm{fixes all three points in} \ E_3(\alpha)[2].
\end{align}
Here, $\beta = \alpha^p$.  But this formula, and the fact that $\alpha, \beta, x_i \in \mathbb{F}_{p^2}$, implies that each $x_i$ is a square in $\mathbb{F}_{p^2}$.
Now,
$$\frac{x_1}{x_2} = \frac{(t+2)^2}{(t+6)^2}$$
implies that $\frac{t+2}{t+6}$ and $t$ lie in $\mathbb{F}_{p^2}$.

\begin{lem}
We have that $t^p = f_j(t)$, for some $j \in \{1,2,4, 5\}$, where
\begin{align*}
&f_1(t) = \frac{-(6t - 4)}{(3t + 6)}, \ \ f_2(t) = \frac{-(12t + 52)}{(3t + 12)}, \ \ f_3(t) = \frac{-(18t + 52)}{(3t + 6)},\\
&f_4(t) = \frac{-(18t + 92)}{(3t + 18)}, \ \ f_5(t) = \frac{-(12t + 44)}{(3t + 12)}, \ \ f_6(t) = \frac{-(6t + 52)}{(3t + 18)}.
\end{align*}
\label{lem:5}
\end{lem}

\begin{proof}
From the fact that $\alpha^{3p} = \beta^3$ for the solutions of $Fer_3$ that occur in \cite{Mor11} and \cite{Mor14}, we see that
$$\alpha^{3p} = \beta^3 = \frac{27\alpha^3}{\alpha^3-27} = \frac{27A(t)}{A(t)-27}, \ \ A(t) = \frac{(3t^2+24t+52)^3}{(t+2)^2(t+4)^2(t+6)^2}.$$
Hence,
$$\alpha^{3p} = A(t^p) = \frac{27A(t)}{A(t)-27} = \frac{27(3t^2 + 24t + 52)^3}{4(3t + 14)^2(3t + 10)^2} = B(t).$$
Now setting $u = t^p$, we factor $A(u)-B(t)$ and find that it splits into linear factors in $u$, so that $u = f_j(t), 1 \le j \le 6$ are the only roots.
Hence $t^p = f_j(t)$ for some $j$.  We also calculate that the linear fractional maps $f_j(t)$, for $j \in \{1,2,4,5\}$, are the only maps with 
order $2$, whereas $f_3(t), f_6(t) = f_3^{-1}(t)$ both have order $6$.  We find further that
$$f_3^{2}(t) - t = -\frac{3t^2+24t+52}{3t+10}, \ \ f_6^{2}(t) - t = -\frac{3t^2+24t+52}{3t+14},$$
so that neither $f_3^2$ nor $f_6^2$ have fixed points, by \eqref{eqn:5.4}, since $\alpha \neq 0$.  Hence $u = t^p = f_j(t)$ for some $j$ listed in the lemma.
\end{proof}

\begin{lem}
We have the formulas
\begin{align*}
\frac{\alpha x_1+3}{\alpha x_1} &= \frac{-4(3t+14)}{3t^2+24t+52},\\
\frac{\alpha x_2+3}{\alpha x_2} &= \frac{4(3t+10)}{3t^2+24t+52},\\
\frac{\alpha x_3+3}{\alpha x_3} &= \frac{-(3t+10)(3t+14)}{3t^2+24t+52}.
\end{align*}
\label{lem:6}
\end{lem}

\begin{proof}
This follows by direct calculation from \eqref{eqn:5.4} and \eqref{eqn:5.6}--\eqref{eqn:5.8}.
\end{proof}

In the following lemma we determine $t^p$ more precisely, assuming that $\mu$ acts as a transposition on $E_3(\alpha)[2]$.

\begin{lem} \begin{enumerate}[(a)]
\item If $x_1^\mu = x_2$, then $t^p = f_5(t)$.

\item If $x_1^\mu = x_3$, then $t^p = f_4(t)$.

\item If $x_2^\mu = x_3$, then $t^p = f_1(t)$.
\end{enumerate}
\label{lem:7}
\end{lem}

\begin{proof}
(a) The assumption implies that
$$x_1^\mu = -\left(\frac{\beta}{6}\frac{\alpha x_1 + 3}{\alpha x_1}\right)^{2p} = x_2,$$
or that
$$\frac{-\alpha^2}{36} \left(\frac{-4(3t+14)}{3t^2+24t+52}\right)^{2p} = \frac{-\alpha^2(t + 4)^2(t + 6)^2}{(3t^2 + 24t + 52)^2},$$
which is equivalent to
$$\left(\frac{-4(3t+14)}{3t^2+24t+52}\right)^{p} = \pm \frac{6(t + 4)(t + 6)}{(3t^2 + 24t + 52)}.$$
Assume that $t^p = f_j(t)$, for $j \in \{1,2,4\}$.  Applying this to the left side of the last equation gives that
\begin{align*}
j = 1, 2:& \  -\frac{6(t+4)(t+2)}{3t^2+24t+52} = \pm \frac{6(t + 4)(t + 6)}{(3t^2 + 24t + 52)} \\
& \Rightarrow \ t+2 = \pm (t+6) \ \Rightarrow \ t = -4;\\
j = 4:& \ \frac{3(t + 6)(t + 2)}{3t^2 + 24t + 52} = \pm \frac{6(t + 4)(t + 6)}{(3t^2 + 24t + 52)}\\
& \Rightarrow \ t+2 = \pm 2(t+4) \ \Rightarrow \ t= -6, -\frac{10}{3}.
\end{align*}
Since these values are excluded, we must have $t^p = f_5(t) = \frac{-(12t + 44)}{(3t + 12)}$, by Lemma \ref{lem:5}.  Checking, we find that
$$\left(\frac{-4(3t+14)}{3t^2+24t+52}\right)^{p} = - \frac{6(t + 4)(t + 6)}{(3t^2 + 24t + 52)},$$
identically.
\medskip

(b) Now suppose that $x_1^\mu = x_3$, or that
$$\left(\frac{-4(3t+14)}{3t^2+24t+52}\right)^{p} = \pm \frac{3(t + 2)(t + 6)}{(3t^2 + 24t + 52)}.$$
If $t^p = f_j(t)$, $j \in \{1,2\}$, we find that $t = -2$ or $-\frac{14}{3}$, which are excluded; and if $j = 5$, then $t = -6$ or $-\frac{10}{3}$.  Thus we
must have $j = 4$, in which case
$$\left(\frac{-4(3t+14)}{3t^2+24t+52}\right)^{p} = \frac{3(t + 2)(t + 6)}{(3t^2 + 24t + 52)}.$$
\medskip

(c) The condition $x_2^\mu = x_3$ is equivalent to
$$\left(\frac{4(3t+10)}{3t^2+24t+52}\right)^{p} = \pm \frac{3(t + 2)(t + 6)}{(3t^2 + 24t + 52)},$$
and as before, we find that only $t^p = f_1(t)$ is possible.
\end{proof}

\begin{thm}
If the map $\mu$ in \eqref{eqn:5.9} on the nontrivial points of $E_3(\alpha)[2]$ is a transposition, 
then for the quadratic character of $3t^2+24t+52$ in $\mathbb{F}_{p^2}$, we have
\begin{equation*}
\psi_p(3t^2+24t+52) = (3t^2+24t+52)^{(p^2-1)/2} = \left(\frac{3}{p}\right).
\end{equation*}
\label{thm:15}
\end{thm}

\begin{proof}
Assume first that $\mu = (x_1 \ x_2)$ as a permutation on $\{x_1, x_2, x_3\}$.  Let $\eta = 3t^2+24t+52$.
 From Lemma \ref{lem:7}(a), we have that $t^p = f_5(t)$, which gives that
 $$\eta^p = \frac{4(3t^2 + 24t + 52)}{3(t + 4)^2},$$
 so that
 $$\eta^{p-1} = \frac{4}{3(t+4)^2} \ \Rightarrow \ \eta^{(p-1)/2} = \frac{2}{\sqrt{3} (t+4)},$$
 for some square-root of $3$ in $\mathbb{F}_{p^2}$.  Now, letting $\varepsilon = \left(\frac{3}{p}\right)$, we have $\sqrt{3}^p = \varepsilon \sqrt{3}$
 and
 \begin{align*}
 \eta^{(p^2-1)/2} &= \left(\frac{2}{\sqrt{3} (t+4)}\right)^{p+1}\\
 &= \frac{2}{\varepsilon \sqrt{3} (t^p+4)} \frac{2}{\sqrt{3} (t+4)}\\
 &=  \varepsilon \frac{4}{3(t+4)(f_5(t)+4)}\\
 &= \varepsilon,
 \end{align*}
since $f_5(t)+4 = \frac{4}{3(t+4)}$.  \medskip

Second, if $\mu = (x_1 \ x_3)$, then $t^p = f_4(t)$ and
$$\eta^p = \frac{16(3t^2 + 24t + 52)}{3(t + 6)^2},$$
giving
$$\eta^{(p-1)/2} = \frac{4}{\sqrt{3} (t+6)}.$$
Now we find
$$\eta^{(p^2-1)/2} = \varepsilon \frac{16}{3(t+6)(f_4(t)+6)} = \varepsilon,$$
since $f_4(t)+6 = \frac{16}{3(t+6)}$. \medskip

Finally, if $\mu = (x_2 \ x_3)$, then $t^p = f_1(t)$ and 
$$\eta^p = \frac{16(3t^2 + 24t + 52)}{3(t + 2)^2},$$
giving
$$\eta^{(p-1)/2} = \frac{4}{\sqrt{3}(t+2)}.$$
Hence,
$$\eta^{(p^2-1)/2} = \varepsilon \frac{16}{3(t+2)(f_1(t)+2)} = \varepsilon,$$
since $f_1(t)+2 = \frac{16}{3(t+2)}$.  This completes the proof.
\end{proof}

\begin{thm}
If the map $\mu$ in \eqref{eqn:5.9} on the nontrivial points of  $E_3(\alpha)[2]$ is a transposition, then
$$\psi_p(\alpha) = \left(\frac{3}{p}\right).$$
Furthermore, in that case $\alpha^3$ and $\alpha^{3p}$ are roots of $x^2+ax-27a = 0$, where $a \in \mathbb{F}_p$ and
$$\left(\frac{a}{p}\right) = \left(\frac{-1}{p}\right).$$
Thus, if
$$q(X) = X^2+rX+s = X^2+(a^3+126a^2+2944a)X + a(a-192)^3$$
is an irreducible factor of $H_{-12p}(X)$ modulo $p$, where $p > 53$ and
$q(X) \notin \{\bar H_{-20}(X), \bar H_{-32}(X), \bar H_{-35}(X)\}$, then $\mu$ is a transposition and
$$\left(\frac{a}{p}\right) = \begin{cases} -1, & \ \textrm{if} \ p \equiv 3 \ (\textrm{mod} \ 4),\\
+1, & \ \textrm{if} \ p \equiv 1 \ (\textrm{mod} \ 4). \end{cases} $$
\label{thm:16}
\end{thm}

\noindent {\bf Remark.} This theorem proves part (a) and half of part (b) of Conjecture \ref{conj:1}.  For the last assertion of the
theorem, it suffices to consider the factors $q(X) \notin \{\bar H_{-20}(X), \bar H_{-32}(X), \bar H_{-35}(X)\}$, 
since the statement was verified for these three factors in Section 4.2.

\begin{proof}
From the formula
$$\alpha^3 =  \frac{(3t^2+24t+52)^3}{(t+2)^2(t+4)^2(t+6)^2}$$
it is clear that $\psi_p(\alpha) = \psi_p(3t^2+24t+52)$, proving the first assertion.  For the second,
we have that
$$-27a = \alpha^3 \alpha^{3p} = \alpha^{p+1} N(\alpha)^2,$$
where $N$ is the norm to $\mathbb{F}_p$.  Hence,
$$\left(\frac{-3a}{p}\right) = \alpha^{(p^2-1)/2} = \left(\frac{3}{p}\right),$$
which implies that
$$\left(\frac{a}{p}\right) = \left(\frac{-1}{p}\right),$$
as claimed.  To prove the remaining assertions, recall that $\frac{1+\mu}{2} \notin \textsf{R}_{-12p} = \mathbb{Z}[\sqrt{-3p}]$, 
the ring of discriminant $-12p$ in $K = \mathbb{Q}(\sqrt{-3p})$, which is the maximal order when $p \equiv 3$ (mod $4$) and the 
order of conductor $f=2$ when $p \equiv 1$ (mod $4$).  If $\mu$ were the identity permutation on $E_3(\alpha)[2]$, then by \eqref{eqn:5.10},
$\frac{1+\mu}{2}$ would inject into the meromorphism ring of $\textsf{K}$, the function field of $E_3(\alpha)$, and necessarily $p \equiv 1$ (mod $4$).  
But then Deuring's lifting theorem \cite[p. 259]{De41} would imply
that the $j$-invariant $j(\alpha)$ of $E_3(\alpha)$, which satisfies $q(j(\alpha)) = 0$ in $\mathbb{F}_{p^2}$,
 would be the reduction (modulo a prime divisor of $p$ in the Hilbert class field of $K$)
of a root of $H_{-3p}(X)$ in characteristic $0$.  Since $j(\alpha)$ is not a root
of any of the polynomials in $\{\bar H_{-20}(X), \bar H_{-32}(X), \bar H_{-35}(X)\}$, it would have to be a root of the product $\Pi_{i \in I}$ in \eqref{eqn:1.4}.
But that gives a contradiction to the assumption on $q(X)$ and the fact that final products in \eqref{eqn:1.4} and \eqref{eqn:1.5} 
are relatively prime modulo $p$.  Thus, by the criterion \eqref{eqn:5.10} on the map $\mu$, $\mu$ is a transposition.
\end{proof}

To prove the first part of Conjecture \ref{conj:1}(b), we now consider the case when $\mu$ is the identity permutation on $\{x_1, x_2, x_3\}$.
In this case $p \equiv 1$ (mod $4$) and the maximal order in $K$ is $\textsf{R}_{-3p} = \mathbb{Z}[\frac{1+\sqrt{-3p}}{2}]$.  By similar computations
as in the proof of Lemma \ref{lem:7}, we must have
$$t^p = f_2(t) = \frac{-(12t+52)}{3t+12}.$$
Now, with $\eta = 3t^2+24t+52$, as before, we have
$$\eta^p = 3f_2(t)^2+24f_2(t)+52 = \frac{4(3t^2 + 24t + 52)}{3(t + 4)^2},$$
which implies that
\begin{align*}
\eta^{(p-1)/2} &= \frac{2}{\sqrt{3} (t+4)},\\
\eta^{(p^2-1)/2} &=  \frac{4}{\varepsilon 3(t+4)(f_2(t)+4)}, \ \ \varepsilon = \left(\frac{3}{p}\right).
\end{align*}
But now $f_2(t)+4 = \frac{-4}{3(t+4)}$, which yields that
$$\psi_p(\eta) = -\varepsilon.$$
Hence, as in the proof of Theorem \ref{thm:16},
\begin{align}
\notag  \psi_p(\alpha) &= -\varepsilon,\\
\label{eqn:5.11} \left(\frac{-3a}{p}\right) &= \alpha^{(p^2-1)/2} = -\left(\frac{3}{p}\right),
\end{align}
which, since $p \equiv 1$ (mod $4$), yields that $\left(\frac{a}{p}\right) = -1$.
This proves the rest of Conjecture \ref{conj:1}(b).

\begin{thm}
Let $p \equiv 1$ (mod $4$).  If $\mu$ is the identity map on $E_3(\alpha)[2]$, then $\alpha^3$ and $\alpha^{3p}$ are roots of $x^2+ax-27a = 0$,
where $a \in \mathbb{F}_p$,
$$\psi_p(\alpha) = -\left(\frac{3}{p}\right) \ \textrm{and} \ \left(\frac{a}{p}\right) =-1.$$
Thus, for the corresponding factor
$$q(X) = X^2+rX+s = X^2+(a^3+126a^2+2944a)X + a(a-192)^3$$
of $H_{-3p}(X)$, when $p > 53$, the condition $\left(\frac{a}{p}\right) = -1$ always holds.
In particular, when $p \equiv 1$ (mod $4$), the quadratic character of $a$ determines whether the factor $q(X)$ of
$H_{-3p}(X) H_{-12p}(X)$ divides $H_{-3p}(X)$ or $H_{-12p}(X)$ modulo $p$.
\label{thm:17}
\end{thm}

Theorems \ref{thm:16} and \ref{thm:17}, together with the discussion in Section 4.2, imply the assertions of Theorem \ref{thm:D}
in the Introduction.  \medskip

While Proposition \ref{prop:7} shows that all roots of $W_{(p-\bar e)/3}(x)$ are squares in $\mathbb{F}_{p^2}$, Theorems \ref{thm:16} and \ref{thm:17}
show that this is not always the case for the roots of $W_{(p-\bar e)/3}(1-x^3/27)$, i.e. for the supersingular parameters for $E_3(\alpha)$, 
at least when $p \equiv 1$ (mod $4$).  When the class number $h(-3p)>6$, there will always be quadratic factors $q(X)$ 
of both $H_{-3p}(X)$ and $H_{-12p}(X)$ over $\mathbb{F}_p$, by \eqref{eqn:1.4} and \eqref{eqn:1.5},
and the corresponding values $\psi_p(\alpha)$ will have opposite signs.

\noindent (karl) Karl Statistical Services, LLC \smallskip

\noindent Aurora, Colorado, 80016 \smallskip

\noindent {\it E-mail}: akarl@asu.edu \medskip

\noindent (morton) Dept. of Mathematical Sciences \smallskip

\noindent Indiana University at Indianapolis (IUI) \smallskip

\noindent 402 N. Blackford St., Indianapolis, Indiana, 46202 \smallskip

\noindent {\it E-mail}: pmorton@iu.edu

\end{document}